\documentclass[12pt]{article} 
\usepackage{amsmath,amssymb,latexsym,amsthm,tikz, enumerate,amscd,hyperref}
\usepackage[abbrev]{amsrefs} 
\usepackage{graphicx} 
\usepackage{todonotes}
\newtheoremstyle{dotless}{}{}{\itshape}{}{\bfseries}{}{ }{}
\usetikzlibrary{decorations.markings}
\newcommand{\image}{\operatorname{Im}}

\newtheorem{Theorem}{Theorem}[section]

\newtheorem{proposition}[Theorem]{Proposition} 
\newtheorem{corollary}[Theorem]{Corollary} 
 
\newtheorem{lemma}[Theorem]{Lemma}

\newtheorem*{theorem*}{Theorem}
\newtheorem*{corollary*}{Corollary}

\newtheorem*{mainthm}{Main Theorem}

\newtheorem*{propA}{Property A}
\newtheorem*{propA'}{Property A$'$}
\newtheorem*{propB'}{Property B$'$}
\newtheorem*{propB}{Property B}
\newtheorem*{propC}{Property C}
\newtheorem*{propD}{Property D}

\theoremstyle{definition}
\newtheorem{definition}[Theorem]{Definition} 
\newtheorem{remark}[Theorem]{Remark}

\newtheorem{example}[Theorem]{Example}

\newtheorem*{question}{Question}

\DeclareMathOperator{\piprod}{\raisebox{-0.1em}{\huge{$\pi$}}\kern -0.2em}

\newcommand{\ca}{\mathcal {A}}

\newcommand{\cac}{\mathcal {C}} 
\newcommand{\calg}{\mathcal{C}_{\operatorname{alg}}} 
\newcommand{\ctop}{\mathcal{C}_{\operatorname{top}}}
\newcommand{\cgro}{\mathcal{C}_{\operatorname{gro}}}

\newcommand{\de}{\operatorname{Del}}
\newcommand{\mo}{\operatorname{Mod}}
 
\newcommand{\ce}{\mathcal {E}}

\newcommand{\ch}{\mathcal {H}} 
\newcommand{\ci}{\mathcal {I}} 

\newcommand{\zci}{\ci^{(0)}}

\newcommand{\cm}{\mathcal {M}} 
\newcommand{\cn}{\mathcal {N}}

\newcommand{\cp}{\mathcal {P}} 
\newcommand{\cq}{\mathcal {Q}}

\newcommand{\cs}{\mathcal {S}}

\newcommand{\cv}{\mathcal {V}} 

\newcommand{\cz}{\mathcal {Z}}
\newcommand{\ant}{\operatorname{Ant}}

\newcommand{\eps}{\epsilon}

\newcommand{\cc}{\mathbb{C}}

\newcommand{\pp}{{\mathbb P}} 
\newcommand{\g}{\textbf{G}} 
\newcommand{\QQ}{{\mathbb{Q}}} 
\newcommand{\rr}{{\mathbb R}}

\newcommand{\zz}{{\mathbb Z}}

\newcommand{\minus}{{-1}}

\newcommand{\ti}{\tilde}

\newcommand{\tB}{{\tilde{B}}}

\newcommand{\ve}{{V/E}} 

\newcommand{\h}{\hat}

\newcommand{\wt}{\widetilde}

\newcommand{\ol}{\overline}

\newcommand{\zero}{\mathbf 0}

\newcommand{\act}{\curvearrowright}

\newcommand{\parab}{\operatorname{Par}}
\newcommand{\iparab}{\operatorname{IPar}}

\newcommand{\Ga}{\Gamma}

\newcommand{\su}{\succcurlyeq}
\newcommand{\pr}{\preccurlyeq}
\def\clap#1{\hbox to 0pt{\hss#1\hss}}

\newcommand{\comment}[1]{} 
 
\newcommand{\ga}{\alpha} 
\newcommand{\gb}{\beta}

\newcommand{\gr}{\rho} 
\newcommand{\gs}{\sigma}

\newcommand{\eq}{\stackrel{*}{=}}

\newcommand{\gD}{\Delta}

\newcommand{\D}[0]{\mathbf{D}}
\newcommand{\G}[0]{\mathbf{G}}

\newcommand{\db}{{\partial}}

\newcommand{\vertex}{\operatorname{Vert}} 
\newcommand{\edge}{\operatorname{Edge}}

\newcommand{\card}{\operatorname{Card}} 
\newcommand{\fan}{\operatorname{Fan}} 
\newcommand{\sfan}{\db\operatorname{Fan}}

\newcommand{\Ker}{\operatorname{Ker}} 
 
\newcommand{\Ima}{\operatorname{Im}}

\newcommand{\support}{\operatorname{supp}} 
\newcommand{\s}{\operatorname{Sal}} 
\newcommand{\Ts}{\operatorname{TSal}}

\newcommand{\type}{\operatorname{type}}

\newcommand{\prj}{\operatorname{Proj}}

\newenvironment{enumerate1}{ 
\begin{enumerate}[\upshape (1)]}
	{ 
\end{enumerate}
} 

\newenvironment{enumeratei}{ 
\begin{enumerate}[\upshape (i)]}
	{ 
\end{enumerate}
} 

\newenvironment{enumeratei'}{ 
\begin{enumerate}[\upshape (i)$'$]}
	{ 
\end{enumerate}
} 

\newenvironment{enumerate1'}{ 
\begin{enumerate}[\upshape (1)$'$]}
	{ 
\end{enumerate}
}

\newenvironment{enumerateI}{ 
\begin{enumerate}[\upshape (I)]}{ 
\end{enumerate}
} 
\newenvironment{enumeratea}{ 
\begin{enumerate}[\upshape (a)]}{ 
\end{enumerate}
}

\newenvironment{enumeratea'}{ 
\begin{enumerate}[\upshape (a)$'$]}{ 
\end{enumerate}
}

\numberwithin{equation}{section} 
\begin{document}
\title{Bordifications of hyperplane arrangements and their curve complexes}

\author{Michael W. Davis 
\and {Jingyin Huang}  
}
\date{\today} \maketitle

\begin{abstract} 
The complement of an arrangement of hyperplanes in $\cc^n$ has a natural bordification to a manifold with corners formed by removing (or ``blowing up'') tubular neighborhoods of the hyperplanes and certain of their intersections.  When the arrangement is the complexification of a real simplicial arrangement, the bordification closely resembles Harvey's bordification of moduli space.   We prove that the faces of the universal cover of the bordification are parameterized by the simplices of a simplicial complex $\cac$, the vertices of which are the irreducible ``parabolic subgroups'' of the fundamental group of the arrangement complement.  So, the complex $\cac$ plays a similar role for an arrangement complement as the curve complex does for moduli space. Also, in analogy with  curve complexes and with spherical buildings, we prove that $\cac$ has the homotopy type of a wedge of spheres. 
\smallskip

\noindent
\textbf{AMS classification numbers}. Primary: 20F36,  20F65, 52N65, 57S30, 
Secondary: 20F55, 32S22, 57N65
\smallskip

\noindent
\textbf{Keywords}: hyperplane arrangement, curve complex, Artin group, bordification
\end{abstract}

\section*{Introduction}
\paragraph{Background.}
A ``compactification'' of a possibly noncompact manifold $X^\circ$ is a compact manifold with boundary $X$ so that that $X^\circ$ is identified with the interior of $X$. For example, $X^\circ$ could be a finite volume hyperbolic manifold  and $X$ the manifold with boundary obtained by adding a flat manifold as a boundary component for each cusp. Let $Y^\circ$  denote the universal cover of $X^\circ$ and $Y$ the universal cover of $X$, then $Y$ is called a ``bordification'' of $Y^\circ$. Put $G=\pi_1(X^\circ)=\pi_1(X)$. There are some classical examples of locally symmetric manifolds and various moduli spaces each of which can be compactified to manifold with corners (in fact, the commpactification has the extra property of being  a ``manifold with faces'' as in Subsection~\ref{ss:tubular}). Here are two basic examples.
\begin{enumerateI}
	\item
	(\emph{The Borel-Serre compactification of an arithmetic manifold}).  Suppose $G=GL(n,\zz)$, $Y^\circ=O(n)\backslash GL(n,\rr)$ and $X^\circ=Y^\circ /G$.  (For $X$ to be a manifold rather than a an orbifold, we actually should replace $G$ by a torsion-free subgroup of finite index.)  Siegel \cite{siegel} described a $G$-equivariant bordification of $Y^\circ$ to a manifold with faces $Y$.  This was then generalized to other arithmetic groups in the foundational paper of Borel-Serre \cite{bs}.
	\item
	(\emph{The Harvey compactification of moduli space}). Suppose $S_{g,n}$ is a surface, possibly with boundary, of genus $g$ and with $n$ punctures, $G=\mo (g,n)$ the corresponding mapping class group, $Y^\circ$ the corresponding Teichm\"{u}ller space and $X^\circ=Y^\circ/G$ the corresponding moduli space.  In  \cite{harvey} Harvey defined a $G$-equivariant bordification $Y$ of $Y^\circ$.
\end{enumerateI}

In both cases $Y$ is a manifold with faces. Let $\cac$ be the nerve of the covering of $\db Y$ by the codimension-one boundary faces of $Y$: the $(k-1)$-simplices of $\cac$ correspond to the boundary faces of codimension $k$ in $Y$. 

In both cases, $\cac$ admits an alternative group theoretic description.  In case (I) $\cac$ can be identified with the Tits building for $GL(n,\QQ)$ (see  \cite[Theorem 8.4.1]{bs}), where Tits building can be defined from the poset of parabolic subgroups of $G$. In case (II), $\cac$ can be identified with the curve complex for the surface $S_{g,n}$ \cite{harvey}. The curve complex be described in terms of subgroups of $G$: each vertex of $\cac$ corresponds to an infinite cyclic subgroup of $G$ generated by a Dehn twist. A $(k-1)$-simplex in $\cac$ corresponds to a family of $k$ commuting Dehn twists $\zz$-subgroups. 

In both cases (I) and (II),  it turns out that the simplicial complex $\cac$ has the homotopy type of a wedge of spheres \cite{solomon,harvey}. Moreover, $Y$ as well as each of its boundary faces is contractible. 

\paragraph{Algebraic and topological versions of curve complexes.}
Let $\ca$ be a finite arrangement of linear hyperplanes in $\mathbb C^n$ and let $\cm(\ca)$ denote the complement of the union of these hyperplanes. We assume for simplicity that $\ca$ is irreducible, i.e. $\ca$ is not the direct sum of two arrangements. Let $X^\circ$ be the intersection of $\cm(\ca)$ with the unit sphere of $\mathbb C^n$. Then $X^\circ$ is homotopic to $\cm(\ca)$. The goal in this paper is to develop a similar picture as in the previous subsection with such $X^\circ$ and $G=\pi_1 (X^\circ)$. This includes  the case when $G$ is a pure  Artin group of spherical type and $\ca$ is the associated arrangement of reflection hyperplanes. More concretely, we want to do  the following. 
\begin{enumeratea}
	\item Define a natural bordification of $X^\circ$ to a manifold with corners $X$, the universal cover of which universal is denoted  $Y(\ca)$.
	\item Define an analog of curve complex $\ctop$ (=$\ctop (\ca)$) as the nerve of the covering of $\db Y$ by its codimension-one boundary faces.
	\item Define another analog of curve complex $\calg$ ($=\calg(\ca)$) from an algebraic viewpoint by looking at commutation of certain $\mathbb Z$-subgroups.
	\item Prove that $\ctop$ and $\calg$ are isomorphic.
	\item Prove that $\calg$ is homotopy equivalent to a wedge of spheres.
\end{enumeratea}
The simplicial complex $\calg$  captures the intersection pattern of certain collection of abelian subgroups in $G$. In  the study of symmetric spaces, mapping class groups and right-angled Artin groups the analogous simplicial complex has  been fundamental to understanding the asymptotic geometry of $G$.

Point (d) provides the link between the topology of the universal cover $Y(\ca)$ and properties of $\calg$, which leads to (e) and other properties of the purely algebraically defined $\calg$. It seems highly non-trivial to establish the above lists for all arrangement. However, all of them can be done for complexification of central real simplicial arrangements studied by Deligne \cite{deligne}, which covers the case where $G$ is a pure Artin group of spherical type. Most of this paper is devoted to proving (d) for these arrangements. In the remainder of the introduction we elaborate $(a)-(e)$.

Let $V=\mathbb C^n$. Suppose $\cq(\ca)$ denotes the set of  subspaces of $V$ that are equal to an intersection  of hyperplanes in $\ca$.  The set $\cq(\ca)$, partially ordered by inclusion, is called the \emph{intersection poset} of $\ca$.   The idea for constructing a bordification of $\cm(\ca)$ as in (a) is to remove tubular neighborhoods of a collection of subspaces  $\cs\subset\cq(\ca)$ where each hyperplane of $\ca$ belongs to $\cs$.  Actually, rather than removing tubular neighborhoods it is better to attach boundary pieces to $\cm(\ca)$ where the boundary piece corresponding to a subspace $E\in\cs$ is  the normal sphere bundle of $E$ restricted to arrangement complement $\cm(\ca^E)$, where $\ca^E$ means the restriction of $\ca$ to $E$.  In order for  this process to yield  a bordification of $\cm(\ca)$ which is independent of various choices, it is necessary for $\cs$ to be a ``building set'' in the sense of DeConcini-Procesi \cite{dp}.  
This construction is carried out by Gaiffi in \cite{gaiffi03}.  

A subspace $E$ is \emph{irreducible} if the normal arrangement $\ca_{V/E}$ to $E$ is irreducible, i.e., if the induced arrangement on $V/E$ does not decompose as a nontrivial direct sum. The bordification of $\cm(\ca)$ with $\cs$ being the set of irreducible subspaces induces a compactification $X$ of $X^\circ$. Motivated by the examples of locally symmetric spaces and moduli spaces, define a simplicial complex $\ctop(\ca)$ to be the nerve of the covering of $\db Y$ by the codimension-one faces of the universal cover $Y$ of $X$. We call this the ``topological curve complex'' of $\ca$.

There  also is an algebraic version of the curve complex that can be defined in terms of commuting families of $\zz$-subgroups of $G$ ($=\pi_1(\cm(\ca)$).  Given a subspace $E\in \cq(\ca)$ with normal arrangement $\ca_{V/E}$ the image of $\pi_1(\cm(\ca_{V/E})$ in $G$ is called a \emph{parabolic subgroup} of $G$ of type $E$ (cf.\ Definition~\ref{d:parabolic} in Section~\ref{s:complements}); the parabolic subgroup is \emph{irreducible} if $E$ is  irreducible. Each irreducible parabolic subgroup has a well-defined central $\zz$-subgroup corresponding to the Hopf fiber coming from the action of $\cc^*$ on $\cm(\ca)$.  Such a $\zz$-subgroup is the analog of the subgroup of $\mo(g,n)$ generated by a Dehn twist. Take these central $\zz$-subgroups of irreducible parabolic subgroups to be the vertex set of the simplicial complex; two vertices are connected by an edge if and only if the corresponding $\zz$-subgroups commute; the \emph{algebraic curve complex} $\calg(\ca)$ is the associated flag complex.  In other words, $\calg(\ca)$ is the flag complex associated to the ``commutation graph'' of these central $\mathbb Z$-subgroups. (see \cite {kimkoberda} for the analogous notion for right-angled Artin groups).  

In the special case when $G$ is a pure spherical Artin group, $\calg$ coincides with the ``complex of irreducible parabolic subgroups'' defined by Cumplido, Gebhardt, Gonz\'alez-Menses and Wiest in \cite{cgw}.


One advantage of defining $\calg$ in the more general setting of hyperplane complements is that the stabilizer of a simplex of $\calg$ is again the fundamental group of $\cm(\ca')$ where $\ca'$ is a ``simpler'' arrangement. We verify this when $\ca$ is the complexification of a real simplicial arrangement, which is convenient for an inductive way of studying $G$. On the other hand, if one restricts to the class of pure spherical Artin groups, then it is not clear that stabilizer of a simplex of $\calg$ is commensurable to another pure spherical Artin group, as the restriction of a Coxeter arrangement might not be a Coxeter arrangement. 

Both $\ctop$ and $\calg$ admit simplicial $G$-actions and in both cases the quotient complex can be characterized in terms of irreducible subspaces of $\ca$ (cf. Definition~\ref{d:I0}). In this generality it seems possible, albeit unlikely, that the topological and algebraic versions of curve complex are always equal. However, we can verify this whenever $\ca$ is the complexification of a central real simplicial arrangements, in which case we also deduce that $\calg$ is homotopy equivalent to a wedge of spheres.

\paragraph{Statements of the results.}
\begin{mainthm}\textup{(Theorem~\ref{thm:main body} in the sequel).}
	Suppose $\ca$ is the complexification in $\cc^n$ of real simplicial arrangement and that $\ca\cong \ca_1\oplus\cdots \oplus \ca_l$ has $l$ irreducible factors.  Let $X$ be the $(2n-l)$-dimensional manifold with corners discussed above and let $Y$ be  its universal cover.  The following statements are true.
	\begin{enumeratei}
		\item
		The algebraic and topologicial versions of the curve complex $\calg$ and $\ctop$ are identical (and we denote this simplicial complex by $\cac$).
		\item 
		The   faces in  $\db X$ are indexed by the set of simplices in a complex $\ci_0$ which is defined in Definition~\ref{d:I0} in Section~\ref{ss:hyperplanes}.  The faces of $Y$ are indexed by $\cac$ and the group $G=\pi_1(X)$ acts on $\cac$ with quotient space $\ci_0$.
		\item
		The simplicial complex $\cac$ is homotopy equivalent to a wedge of $(n-l-1)$-spheres (where $\dim \cac = n-l-1$).
		\item Each face of $X$ is aspherical and its fundamental group injects into $G$.
		\item
		The fundamental group of each codimensional-one face of $X$ is the normalizer of certain parabolic subgroup.
		\item The stabilizer of each simplex of $\cac$
		is the fundamental group of $\cm(\ca')$ where $\ca'$ is also the complexification of some real simplicial arrangement.
	\end{enumeratei}
\end{mainthm}
Let $A_W$ be the spherical Artin group associated to a finite Coxeter group $W$ and let $S$ be the standard generating set of $A_W$.  A \emph{parabolic subgroup} of $A_W$ is any conjugate of a subgroup generated by a subset of $S$. The \emph{pure Artin group} $PA_W$ is  the kernel of $A_W\to W$. 
Associated to the Coxeter group $W$ there is a reflection arrangement $\ca_W$ of hyperplanes in $\cc^n$ so that $\pi_1(\cm(\ca_W))=PA_W$. The bordification of $\cm(\ca)$ is homotopy equivalent to the compact manifold with corners $X =X(\ca_W)$ described as above. Let $Y$ be the universal cover of $X$. The group $W$ acts freely on $X$ and the fundamental group of $X/W$ is $A_W$. So $A_W$ acts on $Y$ by an action that permutes the boundary faces. The arrangement $\ca_W$ is known to be simplicial; so, the Main Theorem can be applied. Hence, $A_W\act \ctop(\ca_W)=\calg(\ca_W)$.

On the other hand, a simplicial complex $\cac(A_W)$ is defined in \cite{cgw} as follows. The vertices of $\cac(A_W)$ are in one-to-one correspondence with the infinite cyclic subgroups that occur as centers of  irreducible, proper parabolic subgroups of $A_W$. Each $(k-1)$-simplex of $\cac(A_W)$ corresponds to a family of $k$ of these $\zz$-subgroups that mutually commute. Since irreducible parabolic subgroups of $A_W$ correspond to parabolic subgroups arising from irreducible subspaces in  $\cq(\ca_W)$, there is an $A_W$-equivariant isomorphism from between $\calg(\ca_W)$ and $\cac(A_W)$.

\begin{corollary*}\label{c:intro}
	Suppose the spherical Artin group $A_W$ has $l$ irreducible factors.  Then $\cac(A_W)$ in \cite{cgw} is homotopy equivalent to a wedge of $(n-l-1)$-spheres. 
\end{corollary*}

Where this paper differs  from \cite{cgw} is that $\cac(A_W)$ is not only defined algebraically as the complex of irreducible parabolics, but it also has a topological interpretation as the nerve of a bordification. This provides us with extra information as in the above corollary.  

\paragraph{Remarks on the braid arrangement $\ca_{n-1}$.} 
Let $M_{0,n+1}$ be the moduli space of complex structure on $S_{0,n+1}$, the $(n+1)$-punctured 2-sphere. Let $\ca_{n-1}$ be the arrangement of type $A_{n-1}$ in $\mathbb C^{n-1}$ (this is the braid arrangement for the braid group on $n$ strands). The following statements are true. 
\begin{itemize}
	\item Start with the Harvey compactification $X'$ of $M_{0,n+1}$ and take its universal cover $Y'$.  Then the nerve of the covering of $\partial Y'$ by its codimension-one faces is isomorphic to the curve complex of $S_{0,n+1}$, 
see \cite{harvey}.
	\item Let $X$ be the bordification of $\cm(\ca_{n-1})$ discussed above (see Definition~\ref{d:BS} for more details). The natural action of $\mathbb C^*=\mathbb C-\{0\}$ on $\cm(\ca_{n-1})$ extends to $C^*\act X$. Then $X'$ and $X/\mathbb C^*$ are diffeomorphic as smooth manifolds with corners, see \cite{Kapranov} and also \cite[Section 4]{gaiffi99}.
	\item The algebraic curve complex for the arrangement $\cm(\ca_{n-1})$ is canonically isomorphic to the curve complex of $S_{0,n+1}$.
\end{itemize}
So, Harvey's result can be reinterpreted as saying that the topological curve complex for $\ca_{n-1}$ as defined above is isomorphic to its algebraic curve complex. 

\paragraph{Discussion of the proof.}
We begin with the isomorphism between $\calg$ and $\ctop$. The starting point (step 0), is to reduce this isomorphism problem to a collection of purely group theoretic properties of $G=\pi_1(X)$.

Start with an arbitrary central arrangement $\ca$ of hyperplanes in $V=\cc^n$ and let $G=\pi_1(\cm(\ca))$.  Given  an irreducible subspace $E$, let $\ca^E$ and $\ca_{V/E}$ denote, respectively,  the restriction of the arrangement to $E$ and its normal arrangement. There are natural homomorphisms $f_1:\pi_1(\cm(\ca^E))\to G$, $f_2:\pi_1(\cm(\ca_{V/E}))\to G$ and $f_3: \pi_1(\cm(\ca^E\times\ca_{V/E}))\to G$ where $f_3=f_1\times f_2$. Here $f_1$ induced by pushing off $\cm(\ca^E)$ from $E$ into $V$ so that it sits inside $\cm(\ca)$, and $f_2$ is obtained by considering arrangement in  a normal disk at a point in $E$. The image of $f_2$ is  an irreducible parabolic subgroup and the central $\mathbb Z$-subgroup of $\image f_2$ corresponds to the Hopf fiber in $\cm(\ca_{V/E})$.  In Section~\ref{ss:curvetop} we define the following four properties.
\begin{enumeratea}
	\item Property A asserts $f_1$ is injective. (Since there is a retraction $\cm(\ca)\to \cm(\ca_{V/E})$,  $f_2$ is always injective.)
	\item  Property B concerns intersections of images of the homomorphisms $f_3$ arising for different irreducible subspaces. 
	\item Property C states that $\image f_3$ is the normalizer of the central $\mathbb Z$-subgroup in $\image f_2$.
	\item Property D characterizes when two central $\mathbb Z$-subgroups commute.
\end{enumeratea} 
In Proposition~\ref{prop:iso} we show that when these properties hold,  there is a canonical isomorphism between $\calg$ and $\ctop$. This works for any central arrangement $\ca$.

The proof of the Main Theorem uses several facts special to real simplicial arrangements.  First, since the arrangements $\ca^E$ and $\ca_{V/E}$ also are complexifications of real simplicial arrangements, it follows from Deligne \cite{deligne} that each of  the arrangement complements $\cm(\ca)$, $\cm(\ca^E)$ and $\cm(\ca_{V/E})$ is aspherical.  Second, Salvetti \cite{s87} defined a finite CW complex that  is homotopy equivalent to $\cm(\ca)$ (this only requires $\ca$ is real).  Finally, in his proof of asphericity in \cite{deligne}, Deligne used ideas of Garside concerning  the word problem in the associated ``Deligne groupoid.'' 

The first step of the proof of the theorem is to reduce Properties  A through D to properties concerning subcomplexes of Salvetti complexes and their Deligne groupoids. (For this step to work we need $\ca$ to be a real arrangement; however,  it need not to be simplicial.)
It is obvious that $\image f_2$ can be represented as the fundamental group of a subcomplex of the Salvetti complex; however, this is less clear for $\image f_1$. We find a certain subset of the Salvetti complex whose fundamental group corresponds to $\image f_1$ (this subset is very close to being a subcomplex).  This reduces the verification Properties A through D to computations in the one-skeleton of the Salvetti complex, which is the underlying graph for the Deligne groupoid.

The second step of the proof is to prove versions of Properties A through D for the Deligne groupoid. (For this step to work we need $\ca$ to be a simplicial arrangement so that Garside theory is available.) A simple reduction shows that we only need to prove Properties A, C and D. For spherical Artin groups, Properties C and D were proved previously in \cite{Parisparabolic} and \cite{cgw}. Our proof of Properties C and D is along the same line as in \cite{Parisparabolic} and \cite{cgw}, however, our treatment is more geometric and works for any simplicial arrangement. We also prove Property A for simplicial arrangements.

We speculate that these properties hold outside the realm of simplicial arrangements.  For example, there are some partial results on Properties A and C for supersolvable arrangements in \cite{Parissupersolvable}.

To prove part (iii) of the Main Theorem we need to show that $\ctop$ (or $\calg$) is a wedge of spheres.
For a general $\ca$ it follows from previous work on the cohomology of $\cm(\ca)$ in \cite{squier, djlo} that $\ctop$ has the same  homology as a wedge of spheres. So, in order to show that $\calg$ is homotopy equivalent to a wedge of spheres, it suffices to show $\calg$ is simply connected whenever its dimension is $>1$.  We prove this in the case of a real simplicial arrangement by constructing a map from the Deligne complex to $\calg$. In the case of right-angled Artin groups, this map is analogous to the map from the right-angled building to the extension graph.  Since the Deligne complex is simply connected, a careful analysis of this map proves that $\calg$ is simply connected.

\paragraph{Structure of the paper.}
In Section~\ref{s:complements} we discuss hyperplane arrangements and their compactifications, and we explain Properties A to D. Section~\ref{sec:background} gives background on zonotopes, Salvetti complexes and Deligne groupoids. Section~\ref{sec:gar} concerns Garside theoretic computations in Deligne groupoids, and completes the argument in the second step of the proof of the Main Theorem. Section~\ref{sec:sc} concerns the proof that $\calg$ is simply  connected when its dimension is $>1$. This is needed in the proof that it is homotopy equivalent to a wedge of spheres. In Section~\ref{sec:sal} the first step of the proof is finished and the results are tied  together.

\paragraph{Acknowledgment.}
This work was partially supported by the grant 346300 for IMPAN from the Simons Foundation and the matching 2015-2019 Polish MNiSW fund. J.H. thanks the Max-Planck Institute for Mathematics (Bonn) for its hospitality during June 2019, where part of the work was done.

\section{Bordifications of complements of hyperplane arrangements}\label{s:complements}
\subsection{Hyperplane arrangements}\label{ss:hyperplanes}
A \emph{hyperplane arrangement}  in a complex vector space $V$ is a collection $\ca$ of linear hyperplanes in $V$.  The set  of linear subspaces that can be obtained as intersections of  elements of $\ca$ is called the \emph{intersection poset}; it is denoted by $\h{\cq}(\ca)$ or simply by $\h\cq$.  The partial order is inclusion.  The ambient space $V$ (corresponding to the empty intersection) is the maximum element of $\h\cq$.  The set of proper subspaces in $\h\cq$ is denoted $\cq$ (or $\cq(\ca))$. The arrangement $\ca$ is \emph{essential} if the  zero subspace $\zero$ lies in $\cq(\ca)$, i.e., if $\zero$ is the intersection of all hyperplanes in $\ca$.  If $\ca$ is essential, then put $\cq_0:=\cq - \mathbf 0$.

The \emph{normal arrangement} to a subspace $E\in \cq(\ca)$ is the hyperplane arrangement $\ca_{V/E}$ in $V/E$ defined by
\begin{equation}\label{e:normal}
\ca_{V/E}:=\{H/E \mid H\in \ca \text{ and $E\leq H$} \}.
\end{equation}
There is also an arrangement $\ca^E$, called the \emph{restriction} of $\ca$ to $E$,  defined by
\begin{equation}\label{e:restriction}
\ca^E:=\{H\cap E\mid H\in \ca \text{ and $E\nleq H$}\}.
\end{equation}
If $E'<E$ is another subspace in $\cq(\ca)$, then the image of $\ca^{E'}$ in $\ca_{E'/E}$  is a \emph{subnormal arrangement} to $E$, i..e.,   It is the arrangement in $E'/E$ defined by
\begin{equation}\label{e:subnormal}
\ca_{E'/E}=\{H/E\mid H\in \ca^{E'} \text{ and $E\leq H$}\}.
\end{equation}

The complement in $V$ of the union of hyperplanes in  $\ca$ is denoted by $\cm(\ca)$. 
Notice that if $SV$ is the sphere of directions in $V$, then we have a homeomorphism
\begin{equation}\label{e:compactcore}
\cm(\ca) \cong (SV\cap \cm(\ca))\times (0,\infty).
\end{equation}

\begin{remark}
	In later sections we will consider arrangements which are obtained by complexifying an arrangement of real hyperplanes in a real vector space (a ``real arrangement'').  For example, any finite Coxeter group has a representation as a linear reflection group on $\rr^n$.   The resulting hyperplane arrangement in $\cc^n$ is called a \emph{reflection arrangement}.
\end{remark} 
\paragraph{Irreducible subspaces.}
Suppose $\ca'$ and $\ca''$ are arrangements in vector spaces $V'$ and $V''$.  Their \emph{direct sum} $\ca'\oplus \ca''$ is the arrangement in $V'\oplus V''$ consisting of hyperplanes which are the form of a sum of a hyperplane in one summand with the other summand, i.e., $\ca'\oplus \ca'':=\{H'\oplus V'', V'\oplus H''\mid H'\in \ca', H''\in \ca''\}$.  An arrangement is \emph{irreducible} if it cannot be decomposed as a nontrivial direct sum.   Any arrangement $\ca$ in $V$ can be decomposed into its irreducible factors $\ca_1,\dots, \ca_k$, where $\ca_i$ is irreducible in $V_i$.  That is to say, $\ca\cong \ca_1\oplus\cdots \oplus \ca_k$ and $V\cong V_1\oplus \cdots \oplus V_k$.  It follows from \eqref{e:compactcore} that
\begin{equation}\label{e:compactcore2}
\cm(\ca) \cong \left(\prod_{i=1}^k SV_i\cap \cm(\ca_i)\right)\times (0,\infty)^k.
\end{equation}

A subspace $E\in \cq(\ca)$ is \emph{reducible} if its normal arrangement in $V/E$ splits as a nontrivial direct sum.    In other words, $E$ is reducible if there are subspaces $E'$, $E''$ in $\cq(\ca)$ such that $E=E'\cap E''$ and  $\ca_{V/E}\cong\ca_{V/E'}\oplus \ca_{V/E''}$.  A subspace $E\in \cq(\ca)$ is \emph{irreducible} if it is not reducible.   For example, any hyperplane $H\in \ca$ is an irreducible subspace.  Any subspace $E\in \cq$ has a decomposition into irreducibles.  This means that there are irreducible subspaces $E_1, \dots, E_k$ such that 
\[
E=E_1\cap\cdots \cap E_k \quad\text{ and }\quad\ca_{V/E}\cong \bigoplus_{i=1}^k  \ca_{V/E_i}.
\]
Up to order, the summands are uniquely determined.

\begin{remark}\label{r:minirreducible}
	Suppose $\ca\cong \ca_1\oplus\cdots \oplus \ca_k$ with $V\cong V_1\oplus \cdots \oplus V_l$ is the decomposition of $\ca$ into irreducibles.  Let $E_i$ be the subspace of the direct sum with $i$-component equal to the zero vector and with $j$-component, $j\neq i$, equal to the subspace $V_j$.  Then $E_i$ is irreducible; moreover, it is a minimal irreducible with respect to the partial order on $\cq(\ca)$.
\end{remark}

\paragraph{Definition of the complex of irreducibles.}\label{p:irred} We shall define a simplicial complex $\ci$ ($=\ci(\ca)$) that encodes the intersection pattern of the irreducible subspaces.  Its vertex set $\zci$ ($=\zci (\ca)$) is the set of irreducible subspaces $E\in \cq$. So, $\zci$ is a subposet of $\cq$. A set of two vertices $\{E,E'\}$ determines a $1$-simplex $\ga$ in $\ci$ in exactly two cases: either $E$ and $E'$ are \emph{comparable} (which means that $E<E'$ or $E'<E$),  or the normal arrangement to $E\cap E'$ is reducible.  This describes a simplicial graph $\ci^1$ that is the $1$-skeleton of the complex which we wish to define.  The \emph{complex of irreducibles} $\ci$ is the flag complex associated to the graph $\ci^1$.  One can directly describe the simplices of $\ci$ as being the ``nested subsets'' of $\zci$ defined below.. 

\begin{definition}\label{d:nested} (cf. \cite{dp}, as well as, \cite{fm}, \cite{dls}*{p.1308}).
	A subset $\ga$ of $\zci$ is \emph{nested} if for any subset $\{E_i\}$ of $\ga$ consisting of pairwise incomparable elements, the intersection $F=\bigcap E_i$ is reducible and $\ca_{V/F}\cong \bigoplus_{i=1}^k  \ca_{V/E_i}$.
	A subset $\ga$ of $\zci$  is the vertex set of a simplex in $\ci$ if and only if it is nested. (We shall often confuse a simplex with its vertex set.)  Call a simplex $\gb\in\ci$ \emph{purely incomparable} if its vertex set consists entirely of incomparable elements.
\end{definition}

\begin{definition}\label{d:I0}
	If $\ca$ is irreducible (i.e., if $\mathbf 0$ is an irreducible subspace in $\cq$), then let $\ci_0(\ca)$ denote the full subcomplex of $\ci(\ca)$ spanned by $\zci-\{\mathbf 0\}$.  In general, if $\ca=\ca_1\oplus \cdots \oplus\ca_l$ is the decomposition of $\ca$ into irreducible factors, then $\ci_0(\ca)$ is defined to be the join: 
	\[
	\ci_0(\ca)=\ci_0(\ca_1) * \cdots *\ci_0(\ca_l) .
	\]
\end{definition}

Since $\ci(\ca_i)$ is equal to the cone, $\ci_0(\ca_i)*\mathbf {0}_i$, where $\mathbf {0}_i$ denotes the cone point, we see that $\ci(\ca)=\ci_0(\ca)*\gD^{l-1}$, where $\gD^{l-1}$ is the $(l-1)$-simplex on $\{\mathbf {0}_1, \dots, \mathbf {0}_l\}$.
The reason that we are interested in $\ci_0(\ca)$ is that  its simplices index the boundary faces of the compact manifold with corners $X$ (cf.\ Definition~\ref{d:compactcore} below.)

\begin{remark}\label{r:fm}
	Feichtner and M\"{u}ller prove in \cite{fm} that $\ci(\ca)$ has a subdivision isomorphic to $\vert \cq(\ca)\vert$ (where $\vert \cq(\ca)\vert$ denotes the geometric realization of the order complex of $\cq(\ca)$).  So, when $\ca$ is irreducible, $\ci_0(\ca)$ is homeomorphic to $\vert\cq_0(\ca)\vert$.  
	A classical result of Folkman \cite{folkman} asserts that $\vert \cq_0(\ca)\vert$ is a wedge of spheres of dimension $n-2$, where $n=\dim V$.  Hence, when $\ca$ is irreducible $\ci_0(\ca)$  is a wedge of spheres of the same dimension.  When $\ca$ has $l$ irreducible factors, it then follows from Definition~\ref{d:I0} that  $\dim \ci_0(\ca)=n-l-1$ and that $\ci_0(\ca)$ is a wedge of $(n-l-1)$-spheres.
\end{remark}

\begin{definition}\label{def:forest}
	Using the poset structure of $\zci$ the vertices of a simplex $\ga$ of $\ci$ can be organized into a forest as follows.   The minimal elements of $\vertex\ga$ are said to be at  \emph{level $0$}.  
	Suppose, by induction, that the notion of level $i$ vertex  has been defined and that $E\in \vertex \ga$ is at level $i$.   The minimal elements of $(\vertex \ga)_{>E}$ are at \emph{level $i+1$}.  Let $\{E_0,\dots, E_k\}$  be this set of minimal vertices. Connect each of the $E_i$ by a directed edge from $E$.
	The directed edges then define a forest on $\vertex{\ga}$.
\end{definition}

\begin{lemma}\label{l:center}
	Suppose $\ca$ is an essential, central hyperplane arrangement with irreducible decomposition $\ca=\ca_1\oplus \cdots \oplus \ca_k$. Put $G=\pi_1(\cm(\ca))$ and let $Z(G)$ be its center. 
	\begin{enumerate1}
		\item
		The center $Z(G)$ contains a free abelian subgroup of rank $k$.
		\item
		If $\cm(\ca)$ is aspherical, then $Z(G)\cong \zz^k$.
	\end{enumerate1}	
\end{lemma}
\begin{question}
	Is it always true that $Z(G)\cong \zz^k$?
\end{question}
\begin{proof}[Sketch of proof of Lemma~\ref{l:center}] The fundamental group of the complement of any central arrangement has a $\zz$ in its center coming from the Hopf fiber. Hence, if $\ca$ has $k$ irreducible factors, then $\zz^k\subset Z(G)$.
	
	Next suppose that $\cm(\ca)$ is aspherical and that $\ca$  is irreducible.  Let $\pp\ca$ be the associated projective arrangement and $\cm(\pp\ca)$ its complement in projective space. Since $\cm(\ca)\cong \cc^*\times \cm(\pp\ca)$, we see that $\cm(\pp(\ca)$ is also aspherical. Also, if $\zz$ is a proper subgroup of $Z(G)$, then the center $C$  of $\pi_1(\cm(\pp\ca))$ is nontrivial.  Since $\cm(\pp\ca)$ is aspherical this implies that its Euler characteristic is $0$.  (If a group of type F has a nontrivial normal abelian subgroup, then its Euler characteristic is $0$.)
	A theorem of Crapo (cf.\ \cite{crapo}) asserts that  if  $\ca'$ if is an arrangement of affine hyperplanes and if the complement $\cm(\ca')$ has Euler characteristic $0$, then $\ca'$ decomposes as $\ca''\times \ca_1$,  where $\ca_1$ is is a nontrivial central arrangement (cf.\ \cite{dls}*{section 5.3}).  Since $\cm(\pp\ca)$ is homeomorphic to the complement of the affine arrangement $\ca'$ obtained by regarding one of the hyperplanes of $\pp\ca$ as being a hyperplane at $\infty$, Crapo's Theorem implies that $\ca'$ splits off another central arrangement $\ca_1$.  But this contradicts the assumption that $\ca$ is irreducible.  Hence, the Euler characteristic of $\cm(\ca')$ must be $0$; so, $C$ must be trivial.   The case when  $\ca$ is irreducible immediately implies (2) in  the general case when $\ca$ has more than one irreducible factors.
\end{proof}

\paragraph{Definition of parabolic subgroups and the algebraic curve complex.} 
Let $E\in \cq(\ca)$ be a subspace.
Let $D$ be a small disk about $E/E$ in $V/E$.  Then $D\cap  \cm(\ca_{V/E})$ is homeomorphic to $\cm(\ca_{V/E})$.  Since $D$ can regarded as a normal disk to $E$ at a point in $\cm(E^\ca)$, we see that when $D$ is sufficiently small $D\cap  \cm(\ca_{V/E})$ is a subset of $\cm(\ca)$. Composing the inclusion with the inverse of a homeomorphism
$D\cap  \cm(\ca_{V/E})\cong \cm(\ca_{V/E})$ we get $i:\cm(\ca_{V/E})\to \cm(\ca)$.  The set of hyperplanes $\ca_{V/E}$ can be identified with the set $\ca_E:=\{H\in \ca\mid E\le H\}$ and $\cm(\ca_E)$ is homeomorphic to $E\times \cm(\ca_{V/E})$.  Since $\ca_E\subset \ca$, we have an inclusion $\cm(\ca)\hookrightarrow \cm(\ca_E)$. The composition of this inclusion with the natural projection $\cm(\ca_E)\to\cm(\ca_{V/E})$ gives the retraction $r:\cm(\ca)\to \cm(\ca_{V/E})$. The following is clear.

\begin{lemma}\label{l:normalarrange}
	The composition $\cm(\ca_{V/E})\stackrel{i}{\to} \cm(\ca)\to\cm(\ca_E)\to\cm(\ca_{V/E})$ is homotopic to the identity map. Thus $i:\cm(\ca_{V/E})\to \cm(\ca)$ is $\pi_1$-injective.
\end{lemma}

\begin{corollary}\label{c:aspherical}
	If $\cm(\ca)$ is aspherical, then so is $\cm(\ca_{V/E})$ for any $E\in \cq$.
\end{corollary}

\begin{definition}\label{d:parabolic}
	Let $i:\cm(\ca_{V/E})\to\cm(\ca)$ be the inclusion map  defined above. Take a base point $x\in\cm(\ca)$ and $x'\in\cm(\ca_{V/E})$. Let $\gamma$ be a path from $x\to i(x')$. A \emph{parabolic subgroup} of $G=\pi_1(\cm(\ca),x)$ of \emph{type $E$} is a subgroup that is conjugate to $\gamma (i_{\ast}(\pi_1(\cm(\ca_{V/E})),x'))\gamma^{-1}$. The parabolic subgroup is \emph{irreducible} (resp. \emph{proper}) if the subspace $E$ is irreducible (resp. if $E\neq\{0\}$). 
\end{definition}

\begin{example} 
	If $\ca$ is a finite reflection arrangement, then $G$ is the corresponding pure spherical Artin group. If $E$ is the fixed subspace of a subset of the Artin generators, then, up to conjugation,  the parabolic subgroup $P_E$ associated with $E$ coincides the usual notion of a parabolic subgroup of $G$, i.e., the intersection with $G$  and  the  subgroup of  the Artin group generated by this subset of the generators (cf.\ \cite{cgw}).  The parabolic subgroup $P_E$ is \emph{irreducible} if the normal representation $\ca_{V/E}$ is irreducible or equivalently,  if the Coxeter group corresponding to $E$ has a connected Coxeter diagram (cf.\ \cite{dh17}). 
\end{example}

Let $\parab(G)$ denote the set of parabolic subgroups of $G$ and let $\iparab (G)$ be the subset of irreducible parabolics. 
For any irreducible parabolic subgroup $P$ of type $E$, define its \emph{central $\zz$-subgroup} to be
\[
Z_P:=\Ker[\pi_1(\cm(\ca_\ve))\to \pi_1(\cm (\pp \ca_\ve))].
\]
In other words, $Z_P$ is the subgroup of $\pi_1(\cm(\ca_\ve))$ corresponding to the Hopf fiber.  By Lemma~\ref{l:center}, if $\cm(\ca)$ is aspherical, then $Z_P$ is equal to the center of $P$.

Let $H_1=H_1(\cm(\ca),\mathbb Z)$. Then $H_1$ has a basis of form $\{e_H\}_{H\in \ca}$ where $e_H$ is the $1$-cycle corresponding to a positively oriented loop around a hyperplane $H$.  If $Z_P$ is the central $\zz$ subgroup of an irreducible parabolic subgroup $P$ of type $E$, then its image in $H_1$ can be written as an integral combination of basis elements: $\sum n_H e_H$. In fact, $n_H$ is either $1$ or $0$ depending on whether or not the hyperplane $H$  contains $E$. Putting $\support Z_P=\{H\in \ca\mid n_H\neq 0\}$ we see that $\support Z_P$ depends only on the conjugacy class of $Z_P$ (or of $P$).   Moreover, since $\support Z_P=\ca_E$, the conjugacy class of $Z_P$ determines the subspace $E$. So, we have established the following lemma.

\begin{lemma}\label{l:ZE}
	Suppose $Z$ and $Z'$ are the central  $\zz$ subgroups of irreducible parabolic subgroups $P$ and $P'$.  If $Z$ and $Z'$ are conjugate in $G$, then $\type P=\type P'$. In particular, each central $\mathbb Z$-subgroup have a well-defined type.
\end{lemma}

\begin{definition}
	\label{d:algebraic curve complex}
	We define a simplicial complex $\calg$, called the  \emph{algebraic curve complex}, as follows. First suppose the arrangement $\ca$ is irreducible. The vertex set of $\calg$ are in one-to-one correspondence with central $\mathbb Z$-subgroup of irreducible proper parabolic subgroups of $G$.  Two vertices $v'$ and $v$ are connected by an edge if the corresponding $\mathbb Z$-subgroups generates a subgroup isomorphic to $\mathbb Z\oplus\mathbb Z$.  This defines the \emph{algebraic curve graph}.  The algebraic curve complex will be the completion of this graph to a flag complex. The algebraic curve complex for a reducible arrangement is the join of the algebraic curve complex of its irreducible components. 
\end{definition}

It follows from Lemma~\ref{l:ZE} that each vertex of $\calg$ has a well-defined type, which is an element in $\zci$.

Now we describe certain collection of edges in $\calg$. Take irreducible subspaces $E$ and $E'$. If $E<E'$, then for  appropriate choices of base points and paths connecting them, we can find parabolic subgroups $P$ and $P'$ of type $E$ and $E'$ respectively such that $P'<P$. This implies that the infinite cyclic subgroups $Z_P$ and $Z_{P'}$ commute (since $Z_P$ commutes with everything  in $P$) and that they generate a free abelian subgroup of rank 2 (since they do so in $H_1$). If $E\cap E'$ is reducible, i.e. $\ca_\ve \oplus \ca_\ve '$ is a decomposition of $\ca_{V/(E\cap E')}$, then for appropriate choices of base points we find parabolic subgroups $P$ and $P'$ such that $Z_P$ and $Z_{P'}$ commute. However, in general it is not clear whether all edges of $\calg$ arise in these two situations.

\subsection{Blowing up subspaces}\label{ss:tubular} \paragraph{Definition of a manifold with faces.} A topological $n$-manifold with boundary is a \emph{smooth manifold with corners} if it is locally differentiably modeled on open sets in $[0,\infty)^n$.  Suppose $X$ is an $n$-manifold with corners. If a point $x\in X$  has local coordinates  $(x_1, \dots, x_n)\in [0,\infty)^n$, then denote by $c(x)$ the number of indices $i$ such that $x_i=0$.  (This number  is independent of the choice of local coordinates.) A ``stratum'' of $X$ of codimension $k$ means a connected component of $\{x\in X\mid c(x)=k\}$.  So, a codimension-zero stratum of $X$ is a component of the interior $X-\db X$ and $\db X$ is the union of the strata of $X$ of codimension $\ge 1$.  A point $x\in X$ lies in the closures of at most $c(x)$  strata of codimension one. Call $X$  a \emph{manifold with faces} if each point $x\in X$ is in the closure of exactly $c(x)$  codimension-one strata.  When this is the case, the closure of a codimension-$k$ stratum is a \emph{face} of $X$ of codimension $k$.    A face of $X$ is itself a manifold with faces.  A covering space of a manifold with faces is naturally a manifold with faces.  A prototypical example of an $n$-manifold with faces is a $n$-dimensional  simple convex polytope $P$ is  where  a ``face'' of $P$ has its usual meaning.  

A closed subset $Y\subset X$ is a \emph{submanifold with faces} if the induced smooth structure on $Y$ is that of a manifold with corners (say, of dimension $m$) and if for  any codimension-$k$ face $F$ of $X$, $Y$ intersects $F$ transversely in a disjoint union of codimension-$k$ faces of $Y$.  it follows that $Y$ is itself an $m$-dimensional manifold with faces. In particular, taking $k=0$ we see that $Y-\db Y$ is an $m$-dimensional submanifold of $X-\db X$.

Let $\cv(X)$ denote the set of codimension-one faces of $X$.  Consider the nerve $\cn$ ($=\cn(X)$) of the cover of $\db X$ by the codimension-one faces of $X$.  It is a simplicial complex with vertex set $\cv(X)$ where a $(k-1)$-simplex $\ga$ of $\cn$ corresponds to a collection of $k$ codimension-one faces with nonempty intersection. 
Any such intersection is a disjoint union of codimension-$k$ faces. There is a related cell complex $\gD(X)$ which differs from $\cn$ in that it can have ``multiple simplices.''  For example, if $X$ has two codimension-one faces $\db_0 X$ and $\db_1 X$ and $\db_0X\cap \db_1X \neq \emptyset$, then $\cn(X)$ consists of a single edge connecting $0$ to $1$; however, in $\gD(X)$ there is an edge for each component of the intersection.  In general, the poset of faces in $\db X$ is anti-isomorphic to the poset of cells in a cell complex $\gD(X)$.  The vertex set of $\gD(X)$ is $\cv(X)$ and each codimension $k$ face determines a $(k-1)$-simplex in $\gD(X)$. So, there is a natural map $\gD(X)\to \cn(X)$ whose restriction to each simplex is a homeomorphism. However, since multiple simplices can have the same vertex set, we see that  $\gD(X)$ is only  a $\gD$-complex in the sense of \cite{hatcher}.

Say that $X$ has the \emph{connected intersection property} if each intersection of codimension-one faces is either empty or contains a single face.  When this holds, the natural map $\gD(X)\to \cn(X)$ is an isomorphism and so, $\gD(X)$ is a simplicial complex. 
Note however, that  it might happen that $X$ has the connected intersection property while a  covering space  $Y$ does not. This is the issue we shall be concerned with in subsection~\ref{ss:curvetop}.

\paragraph{Blowing up subspaces and submanifolds.} In this paragraph we explain a method for canonically  ``removing an open tubular neighborhood'' of a submanifold $Y$ in a manifold $M$ without mentioning the phrase ``$\eps$ neighborhood.''  The resulting manifold with boundary $M\odot Y$ is a bordification of $M-Y$ called the \emph{blowup} of $M$ along $Y$.  We  consider some simple examples of this process.  First consider the case where $M$ is a vector space $V$ and $Y$ is the zero subspace $\bf{0}$.  The sphere in $V$ can be defined as the quotient space $SV:=(V-\mathbf{0})/\rr_+$, where the positive real numbers $\rr_+$ act on $V-\mathbf{0}$ via scalar multiplication.  Define the blowup $V\odot \mathbf{0}$ by
\[
(V\odot \mathbf{0}):= (V-\mathbf{0})\times _{\rr_{+}} [0,\infty) \cong SV\times [0,\infty).
\]
Note that $\db(V\odot \mathbf{0})=SV\times \zero$.  A slight generalization is the case where $M=V$ and $Y=E$ is a linear subspace of $V$.  Then $V-E=E\times (V/E-\mathbf 0)$ and
\begin{equation}\label{e:ve}
(V\odot E):= E\times \{(V/E-\mathbf 0)\times _{\rr_{+}} [0,\infty)\} =E\times S(V/E)\times [0,\infty).
\end{equation}

There is another approach to the definition of the blowup $V\odot E$, which we shall generalize in the next subsection.  Define a map $\gr:V-E\to V\times S(V/E)$ as follows. The first component of $\gr$ is the inclusion $V-E\hookrightarrow V$.  Its second component is the composition of projections $V-E\to (V/E)-(E/E)\to S(V/E)$.  One sees that the closure of the image of $\gr$ is can be identified with the blow up, that is, 
\begin{equation}\label{e:imagr}
V\odot E \cong \ol {\Ima\gr}.
\end{equation}
Essentially, the blowup $V\odot E$ is  the only case with which we need be concerned; however, there are some slightly more general situations.  If $N$ is a vector bundle over a smooth manifold $Y$, then its \emph{sphere bundle} $S N$ is defined by
\(
S N=(N-Y)/\rr_+, 
\) 
where $\rr_+$ acts on the complement of the $0$-section $N-Y$ via fiberwise scalar multiplication.  The \emph{cylinder bundle} $CN$ is then defined by
\[
CN :=(N-Y)\times _{\rr_{+}} [0,\infty) = SN\times [0,\infty).
\]
Thus, $N\odot Y:=CN$ is a bordification of $N-Y$.  In general, suppose $Y$ is a smooth submanifold of a manifold $M$ and that $N$ is its normal bundle.  Our discussion follows \cite{janich}*{} or \cite{d78}*{}. 
By a \emph{tubular map} we mean a diffeomorphism $T:N\to M$ onto an open neighborhood of $Y$ such that 1) the restriction of $T$ to the $0$-section is the inclusion, and 2) at any $y\in Y$, under the natural identification $T_yN = T_yM$, the differential of $T$ is the identity map.   A tubular map $T$ induces a function $T': CN\to (M-Y)\sqcup S N$ and there is a unique structure
of a smooth manifold with boundary on $(M-Y)\sqcup S N$ which agrees with the original structure on $M-Y$ so that $T'$ takes $CN$ onto a collared neighborhood of $S N$.  We denote this manifold with boundary,  $M\odot Y$, and call it the \emph{blow-up} of $M$ along $Y$.  It is diffeomorphic to a complement of an open tubular neighborhood of $Y$ in $M$.  Note that $\db (M\odot Y)= S N$. 

Finally, suppose $Y$ is a submanifold with faces of a manifold with faces $X$.  
Then the previous paragraph goes through \emph{mutatis mutandis}.  First, there is a normal vector bundle $N$ over $Y$ such that for any face $F$ of $X$, the restriction $N\vert_{Y\cap F}$ is the normal bundle of $Y\cap F$ in $F$.  As before, we get a new manifold with faces $X\odot Y$ that has acquired a new codimension-one face, namely, $S N$.  Moreover, each face $F$ of $X$ with $F\cap Y \neq \emptyset$ has a bordification $F\odot (Y\cap F)$  that also has acquired a new codimension-one  face, namely,  $S N\vert_{Y\cap F}$.  Thus, $F\odot (Y\cap F)$ is a submanifold with faces of $X\odot Y$.  This generality will allow us, in the next subsection, to iterate the blowing up procedure to certain sequences of submanifolds.

\subsection{Attaching a boundary to the complement of a hyperplane arrangement}\label{ss:arrangements}
Our goal in this subsection is to construct a bordification of $\cm(\ca)$, where, as usual, $\cm(\ca)$ is the complement of an arrangement $\ca$ of linear hyperplanes in a complex vector space $V$.  The idea is to blow up the subspaces belonging to some subset $\cs$ of  $\cq(\ca)$.  In order for the result to be  a manifold with faces this subset $\cs$ must be a ``building set'' as defined in \cite{dp}, \cite{djs} or \cite{fm}. One of the main requirements for a subset to be a building set is that it contain all the irreducible elements in $\cq(\ca)$.  One possible choice  is to take the building set to be all 
of $\cq(\ca)$.  Another possible choice is to take the building set to be $\zci(\ca)$, the set of  irreducibles in $\cq(\ca)$.  This second choice is the one we make throughout this paper. 
\paragraph{Method 1: the closure of an embedding.}  This method is the easiest  to define.  Its disadvantage is that it is not so clear that it actually results in a smooth manifold with corners.
The simplest definition of the blowup is similar to one in \cite{dp}*{p. 461}, see also \cite{gaiffi03}. Let $\cs$ be a collection of linear subspaces of $V$. Consider the embedding 
\[
\gr:V-(\bigcup_{E\in \cs}E) \to V \times \prod_{E\in \cs}  S(V/E),
\]
where the first component of $\gr$ is the inclusion and the component in $S(V/E)$ is the natural projection (which is defined since $V-E\subset V- (\bigcup_{E\in\cs} E)$). Note that if $\cs$ contains all hyperplanes $H\in \ca$, Then the domain of $\gr$ is $\cm(\ca)$.  In a similar fashion, one can define
\[
\gr_S:SV-(\bigcup_{E\in\cs} SE) \to SV \times \prod_{E\in \cs}  S(V/E).
\]

\begin{definition}\label{d:BS} 
	As in \eqref{e:imagr},
	define $V\odot\cs$ (resp. $SV\odot\cs$) to be the closure of the image of $\gr$ (resp. $\gr_S$). If $\cs=\zci$, we write simply $V_\odot$ (resp. $SV_\odot$) instead of $V\odot\cs$ (resp. $SV\odot\cs$). 
\end{definition}

\begin{remark}
	In \cite{dp} De  Concini and Procesi use complex projective spaces $\pp(V/E)$ rather than spheres $S(V/E)$ so that $\cm(\ca)$ is partially compactified by adding a divisor with normal crossings to $\cm(\ca)$ rather than a boundary. 
\end{remark}

Now we introduce the notion of building set, cf. \cite{gaiffi03}*{Definition 2.1} which goes back to \cite{dp}. Let $\cq(\cs)$ be the collection of subspaces of $V$ formed by intersection of elements of $\cs$. In this paper we will be only interested in the case when $\cq(\cs)=\cq(\ca)$ (i.e. $\cs$ is a subset of $\cq(\ca)$ which contains all the hyperplanes in $\ca$). For subspaces $F,F'\in \cq(\ca)$, put 
\begin{align*}
\cs_F&:=\{E\in\cs\mid F\leq E\},\\  
\cs^F&:=\{F\cap B\mid B\in \cs-\cs_F\},\\ 
\cs^{F}_{F'}&:=\{B\cap F\mid B\in \cs_{F'}-(\cs_{F'}\cap \cs_{F})\}. 
\end{align*}
Choose an inner product on $V$. Then for any $F\in \cq(\ca)$ we can identify with its orthogonal complement $F^\perp$ with $V/F$. The set $\cs$ is a \emph{building set} if for any $E\in\cq(\ca)$, we have a decomposition	
\[
E=E_1\cap\cdots \cap E_k \quad\text{ and }\quad\ca_{V/E}\cong \bigoplus_{i=1}^k  \ca_{V/E_i}.
\]
where $E_1, \dots, E_k$ are the minimal elements of $\cs_F$ (with respect to the partial order on $\cq(\ca)$). It follows that any building set $\cs$ must contain the set of irreducibles $\zci$.  It also follows that $\zci$ is a building set.

The following lemma is an immediate consequence of the definition.
\begin{lemma}
	\label{building}
	Suppose $\cs$ is a building set. For each $E\in \cq(\ca)$, both $\cs^E$ and $\cs^{E^\perp}_E$ are buildings sets (with respect to $\cq(\ca^E)$ and $\cq(\ca_{V/E})$ respectively).  
	\end{lemma}

\begin{Theorem}\label{t:gaiffi}\textup{(Gaiffi \cite{gaiffi03}*{Theorem 4.5}).}
	Suppose $\cs$ is a building set. Then $V\odot\cs$ (resp. $SV\odot\cs$) are smooth manifolds with faces. In particular, when $\cs=\zci$, this applies to $V_\odot$ and $SV_\odot$.
\end{Theorem}

By Theorem~\ref{t:gaiffi} and Lemma~\ref{building}, the blowups $S(V/E)\odot\cs^{E^\perp}_E$, $(SE)\odot\cs^E$ and $E\odot\cs^E$ are smooth manifolds with faces.

Next we describe the strata of $V\odot \cs$ for a building set $\cs$.
Its codimension-zero stratum  is  $\cm(\ca)$. Each element $E\in \cs$ 
gives rise to a codimension-one stratum of $V\odot\cs$ as follows. Let $p:V\odot\cs\to V$ (resp. $p_S:SV_\odot\to SV$) be the projection to the first factor in the definition of $\gr$ (resp. $\gr_S$). Define
\[
\partial_E (V\odot\cs):=\overline{p^{-1}(E-\bigcup_{F\in \cs^E}F)}\,,
\]
and
\[
\partial_E (SV\odot\cs):=\overline{p^{-1}_S(SE-\bigcup_{F\in \cs^E}SF)}\,,
\]

\begin{proposition}\label{p:description} \textup{(Gaiffi \cite[Theorem 5.1]{gaiffi03}).} For any $E$ in the building set $\cs$ there are diffeomorphisms, $\partial_E (V\odot\cs)\cong (S(V/E)\odot\cs_E^{E^\perp}) \times (E\odot\cs^E)$ and $\partial_E (SV\odot \cs)\cong (S(V/E)\odot\cs_E^{E^\perp})\times (SE\odot\cs^E)$. 
\end{proposition}
The diffeomorphism arises as follows. Let $K_1$ be the interior of $E\odot\cs^E$ and $K_2$ be the interior of $S(V/E)\odot\cs_E^{E^\perp}$. By considering small line segments emanating from $K_1$, orthogonal to $E$ and going in the direction of $K_2$, we clearly see a copy of $K_2\times K_1$ in $V_\odot$. The closure of this set gives $\partial_E V\odot\cs$.

From now on we will only be interested in the case where $\cs=\zci$, the collection of irreducible elements in $\cq(\ca)$. For any $F\in \cq(\ca)$, define $F_\odot$ to be $F\odot\cs^F$ and $S(V/F)_\odot$ to be $S(V/F)\odot(\cs^{F^\perp}_F)$.
It is shown in \cite{gaiffi03}*{Section 5} that $\{\partial_E V_\odot\}$ is the set of codimension-one strata of $V_\odot$; moreover, the union is $\partial V_\odot$. By Proposition~\ref{p:description}, $\partial_E V_\odot$ is connected and hence, is a codimension-one face of $V_\odot$.

Next we describe intersections of codimension-one faces that give rise to faces of higher codimension. Let $\alpha$ be a subset of $\ci^{(0)}$. Define $$\partial_\alpha V_\odot:=\bigcap_{E\in \alpha}\partial_EV_\odot\, ,\qquad \partial_\alpha SV_\odot:=\bigcap_{E\in \alpha}\partial_ESV_\odot.$$
Let $K_\alpha$ be the forest as in Definition~\ref{def:forest}. 
For $E\in \vertex \ga$, let $\{E_0,E_1,\ldots,E_k\}$ be vertices of $K_\alpha$ in the next level. Set $\hat E=\bigcap_{i=0}^k E_i$ (if there are no vertices in the next level, then $\hat E=V$). As $E$ is irreducible, $E\subsetneq \hat E$. Define $S(\hat E/E)_\odot$ to be $S(\hat E/E)\odot\cs^{\hat E\cap E^\perp}_E$. Let $E_\alpha$ be the intersection of subspaces of $V$ arising from vertices of $\alpha$ at level 0. Define $(E_\alpha)_\odot$ to be $E_\alpha\odot(\cs^{E_\alpha})$.

Suppose $\alpha$ is nested. By Lemma~\ref{building}, for each $E\in \vertex\ga$, $\cs^{\hat E\cap E^\perp}_E$ is a building set and $\cs^{\cap\,  \alpha(0)}$ is a building set. Hence, $S(\hat E/E)_\odot$ and $(E_\alpha)_\odot$ are smooth manifolds with faces.

\begin{Theorem}
	\label{thm:higher codimension pieces}\textup{(Gaiffi \cite{gaiffi03}*{Theorems 5.2 and 5.3}).}
	The intersection  $\partial_\alpha V_\odot$ is nonempty if and only if $\alpha$ is nested. Moreover, there are diffeomorphisms of smooth manifolds with faces: $\partial_\alpha V_\odot\cong (E_\alpha)_\odot\times\prod_{E\in\alpha} S(\hat E/E)_\odot$ and $\partial_\alpha SV_\odot\cong S(E_\alpha)_\odot\times\prod_{E\in\alpha} S(\hat E/E)_\odot$.
\end{Theorem}

Now we explain the diffeomorphism in Theorem~\ref{thm:higher codimension pieces} in the case of $|\alpha|=2$, general cases are similar \cite{gaiffi03}. Suppose $\alpha=\{E,F\}$ is nested. First case to consider is that $E\subset F$ or $F\subset E$ (assume the former). Use the inner product on $V$, we have the identification $V=E\oplus (F/E)\oplus(V/F)$. Let $S\cm(\cs_E^{F\cap E^\perp})$ denotes the interior of $S(F/E)\odot\cs_E^{F\cap E^\perp}$. Let $x_1,x_2,x_3$ be elements in $\cm(\cs^E)$, $S\cm(\cs_E^{F\cap E^\perp})$, and $S\cm(\cs_F^{F^\perp})$ respectively. Consider the curve in $V$ represented by $t\to x_1+tx_2+t^2 x_3$, where we think $x_2$ (resp. $x_3$) as a unit vector in $F/E$ (resp. $V/F$). When $t$ is small enough, this curve is in $\cm(\ca)$, and as $t\to 0^+$, we obtain a point in $\partial (V\odot\cs)$. This gives a injective map 
$$
\cm(\cs^E)\times S\cm(\cs_E^{F\cap E^\perp})\times S\cm(\cs_F^{F^\perp})\to \partial (V\odot \cs)
$$
The closure of the image of this map turns out to be $\partial_E (V\odot \cs)\cap \partial_F (V\odot \cs)$ \cite{gaiffi03}. The second case is that $E\nsubseteq F$ and $F\nsubseteq E$. Let $U=E\cap F$. Then $\ca_{V/U}\cong \ca_{V/E}\oplus \ca_{V/F}$. The inner product of $V$ induces $V=U\oplus E^\perp\oplus F^\perp=U\oplus (V/E)\oplus (V/F)$. Consider the curve in $V$ represented by $t\to x_1+tx_2+tx_3$, where $x_1\in \cm(\cs^U)$, $x_2\in S\cm(\cs^{E^\perp}_E)$ and $x_3\in S\cm(\cs^{F^\perp}_F)$, which gives a map
$$
\cm(\cs^U)\times S\cm(\cs_E^{E^\perp})\times S\cm(\cs_F^{F^\perp})\to \partial (V\odot \cs)\,.
$$
The closure of the image of this map is $\partial_E (V\odot \cs)\cap \partial_F (V\odot \cs)$. 

We summarize part of the above discussion as follows.
\begin{lemma}
	\label{lem:inductive}
	For each $E$ in the building set $\cs$, the arrangement of subspaces $\cs^{E^\perp}_E\oplus \cs^E$ in $\cq(\ca_{V/E}\oplus\ca^E)$ is also a building set. The face $\partial_E (V\odot \cs)$ can be naturally identified with a codimension one face of $((V/E)\oplus E)\odot (\cs^{E^\perp}_E\oplus \cs^E)$. For a nested subset $\{E,F\}$ of $\cs$, $\partial_E (V\odot \cs)\cap \partial_F (V\odot \cs)$ can be identified as a codimension two face of $((V/E)\oplus E)\odot (\cs^{E^\perp}_E\oplus \cs^E)$.  
\end{lemma}

\paragraph{Method 2: iterated blowups.}
We now describe an alternative definition of $V_\odot$. Linearly order the elements of $\zci$: $E_1,\dots, E_p$  in some fashion compatible with the partial order (here $p=\card \zci$).    In other words, $E_i<E_j \implies i<j$. The idea is to blow-up $V$ along  the $E_i$ in succession. We shall inductively define  a sequence of manifolds with faces:  
\[
V=V_{\odot}(1)\,,V_{\odot}(2)\,,\dots ,V_{\odot}(p)=V_\odot
\]
where, roughly speaking, $V_{\odot}(k)$ is obtained by blowing up $V_{\odot}(k-1)$ along $E_k$.  More precisely, let $\ol{E_k}$ denote the closure of $\cm(\ca^{E_k})$ in $V_{\odot}(k-1)$.  (Since $E_k$ has not yet been removed from $V_{\odot}(k-1)$, the restriction arrangement complement $\cm(\ca^{E_k})$ is  a submanifold in the interior of $V_{\odot}(k-1)$.)  By induction we can assume that $\ol{E_k}$ is a submanifold with faces of $V_{\odot}(k-1)$.
Define
\(
V_{\odot}(k)= V_{\odot}(k-1) \odot \ol{E_k},
\)
where the blowup along $\ol{E_k}$ is as defined in the final paragraph of  subsection~\ref{ss:tubular}.  The \emph{bordification} of $\cm(\ca)$ is the manifold with faces defined by 
\[
V_\odot:=V_{\odot}(p).
\]
The bordification is well-defined modulo the following two issues:
\begin{enumerate1}
	\item for each $k$, we need to verify $\ol{E_k}$ is a submanifold with faces of $V_{\odot}(k-1)$;
	\item choosing a different linear order on $\zci$ will result the same $V_\odot$ (up to diffeomorphism between manifolds with corners).
\end{enumerate1}
We refer to  \cite{gaiffi03}*{Sections 7-10} for the details of (1) and (2). The arguments rely on the fact that $\zci$ is a building set. (As before, the results of \cite{gaiffi03} are proved in the more general setting of blowing up along building sets.) For (2), it is shown in \cite[Proposition 10.2]{gaiffi03} that the identity map between the interior of the two blowups arising from different linear orders of $\zci$ extends to a diffeomorphism of the blowups. One can use Method 2 to give an independent description of Proposition~\ref{p:description} and Theorem~\ref{thm:higher codimension pieces} (see \cite[Section 11]{gaiffi03}). We leave the equivalence of Methods 1 and 2 to the reader.

\begin{remark}\label{r:concept}(\emph{Deleting tubular neighborhoods}).
	Although Methods 1 and 2 might  seem  to be complicated, the  concept underlying both methods is simple.  As explained in subsection~\ref{ss:tubular}, given a submanifold $Y$ of a manifold $M$, the blowup $M\odot Y$ is diffeomorphic to the complement of a open tubular neighborhood of $Y$ in $M$.  In this subsection we have started with a more complicated situation, a collection $\cs$ of linear subspaces.  Under either Methods 1 or 2 the blowup is diffeomorphic to the complement of an open regular neighborhood of the union of subspaces in $\cs$.  One can try to remove the tubular neighborhoods directly, one at a time. Although this method is the easiest to visualize,   technical difficulties could be encountered.  For example, one needs to choose carefully radial functions on the normal bundles in order to specify the size of tubular neighborhoods to be removed.  Also, in order for the blowup to be a well-defined manifold with corners we need $\cs$ to be a building set.  The point is that if a stratum is the transverse intersection of subspaces, then there is no reason to remove it first: it will be deleted when we remove the tubular neighborhoods of the subspaces and since their normal sphere bundles will intersect transversely we will get a manifold with corners.  However, if the stratum is not a transverse intersection of subspaces then we must remove it first. This will have the effect that the intersections of various subspaces with the normal sphere bundle of the stratum will be disjoint submanifolds.  So, the elements of $\cs$ that are not transverse intersections must be blown up first.  This is exactly what it means for $\cs$ to be a building set. Finally, one removes neighborhoods of the strata in the same order as was done in Method 2.  Our conclusion is that blowing up can be accomplished by the naive method of removing tubular neighborhoods provided that $\cs$ is a building set.
\end{remark}

\begin{definition}\label{d:compactcore}(\emph{The compact core}).
	Suppose $\ca$ has an irreducible decomposition $\ca=\ca_1\times\cdots\times \ca_\ell$ and corresponding vector space decomposition $V=V_1\oplus \cdots \oplus V_\ell$. Then $V_\odot = (V_1)_\odot \times\cdots\times (V_\ell)_\odot=[0,\infty)^\ell\times S(V_1)_\odot \times \cdots\times S(V_\ell)_\odot$ (cf.\ \eqref{e:compactcore2}).  The \emph{compact core} $X$ (or $X(\ca)$) of $V_\odot$ is defined by omitting the interval factors, i.e.,  
	\[
	X=\prod_{i=1}^\ell S(V_i)_\odot \ .
	\]
\end{definition}

If $X^\circ:=\cm(\ca) \cap \prod S(V_i)$, then $X$ is a bordification of $X^\circ$.  So, $X$ is homotopy equivalent to $V_\odot$ (or to $\cm(\ca)$).
The faces of $S(V_i)_\odot$ of codimension $\ge 1$ are indexed by the simplices of $\ci_0(\ca_i)$ whose vertices are the nonzero irreducible subspaces in $V_i$. In general define $\ci_0(\ca)$ by 
\[
\ci_0(\ca)\cong \ci_0(\ca_1)*\cdots *\ci_0(\ca_\ell).
\]   
Since $X(\ca)=\prod_{i=1}^\ell X(\ca_i)$, the faces of $X$ are indexed by $\ci_0(\ca)$. 
So, the face structure on $X$ reduces to the case where $\ca$ is irreducible.  When $\ca$ is irreducible a codimension-one face of $X$ has the form  $\db_E X =SE_\odot\times S(V/E)_\odot$.  When $\ca$ has more than one factor, a codimension-one face of $X$ has one factor of the form $\db_E X(\ca_i)$ and the others are of the form $X(\ca_j)$, with $j\neq i$. 

\subsection{The curve complex}\label{ss:curvetop}
Let $Y$ denote the universal cover of $X$ and $\pi:Y\to X$ the covering projection.  Then $Y$ is a manifold with faces. Each face of $Y$ is a connected component of $\pi^\minus(\db_\ga X)$ for some $\ga\in \ci_0$. 

The goals in this subsection are  (1) to describe the bordification $Y$ of the universal cover of $X^\circ$, (2) to give the definition of the topological curve complex $\ctop$ associated to the manifold with faces $Y$ (cf.\ Definition~\ref{d:curvetop} below) and (3) to describe some properties of $Y$ which are needed  to insure that $\ctop$ is naturally isomorphic to $\calg$ in Definition~\ref{d:algebraic curve complex}.

As in subsection~\ref{ss:tubular}, let $\cn(X)$ be the nerve of the cover  of $\db X$ by the set of its codimension-one faces, i.e., by
$\{\db_E X\}_{E\in \vertex \ci_0}$.  By Theorem~\ref{thm:higher codimension pieces}, the $\cn(X)=\ci_0$.  It also follows that $X$ has the connected intersection property (i.e., each nonempty intersection of codimension-one faces is connected). Hence, the $\gD$-complex $\gD(X)$ is also equal to the simplicial complex  $\ci_0$.

\begin{definition}\label{d:curvetop}
	The \emph{topological curve complex} $\ctop$ ($=\ctop(\ca)$) is the $\gD$-complex $\gD(Y)$ as in subsection~\ref{ss:tubular}.  In other words, a $(k-1)$-simplex in $\ctop$  corresponds to a codimension-$k$ face of $Y$.
\end{definition}

Each vertex of $\ctop$ has a well-defined \emph{type}, which is a subspace in $\zci$.

We have $\db_E X=SE_\odot \times S(V/E)_\odot$.  By Lemma~\ref{l:normalarrange}, the composition of  inclusions 
$S(V/E)_\odot \hookrightarrow \db_EX\hookrightarrow  X$ is $\pi_1$-injective and the image of $\pi_1(S(V/E)_\odot$ in $G$ ($=\pi_1(X)$) is an irreducible parabolic subgroup of type $E$.  The composition of  inclusions $SE_\odot \hookrightarrow \db_EX\hookrightarrow  X$ induces a homomorphism $\pi_1(SE_\odot)\to G$. 

\paragraph{Properties A, B, C, D.}  
Next we introduce four group theoretic properties that can be used to guarantee that $\ctop$ and $\calg$ are isomorphic. These properties will be proved later for complexifications of real simpliciial arrangements, although we expect the properties to hold in greater generality.

\begin{propA}\label{propA}
	For each $E\in \cq(\ca)$, the inclusion $SE_\odot \hookrightarrow SE_\odot \times S(V/E)_\odot=\db_E X\hookrightarrow X$ is $\pi_1$-injective.
\end{propA}
By Theorem~\ref{thm:higher codimension pieces}, each face $\db_\ga X$ is a product of blown up spheres $S(E'/E)_\odot$ in subnormal arrangements.  By Lemma~\ref{l:normalarrange}, $S(E'/E)_\odot\hookrightarrow S(E')_\odot$ is a retract and hence,  is $\pi_1$-injective.  Property A states that $S(E')_\odot\hookrightarrow X$ is $\pi_1$-injective.  So, $\pi_1(S(E'/E)_\odot)\to \pi_1(X)$, being the composition of two injections, also is injective. Therefore, Property A implies the following.

\begin{propA'}\label{propA'}
	For each $\ga\in \ci_0$, the inclusion $\db_\ga X \hookrightarrow X$ is $\pi_1$-injective.
\end{propA'}

\begin{propB}\label{propB}
	For each simplex $\alpha$ of $\ci_0$, we have $$\pi_1 (\partial_\alpha X,x)=\bigcap_{v\in \vertex \alpha}\pi_1(\db_v X,x)$$ for a base point $x\in \partial_\alpha X$.
\end{propB}

\begin{propB'}\label{propB'}
	Suppose $\{\db_vY\}$  is a collection of  codimension-one faces  of $Y$ such that the intersection of any subcollection is nonempty.  Then $\{\db_vY\}$ has the connected intersection property (see subsection~\ref{ss:tubular}).
\end{propB'}

The fibration of $S(V/E)$ by $\mathbb S^1$ naturally extends to a a fibration of $S(V/E)_\odot$ by $\mathbb S^1$, which we will call the \emph{Hopf fibration} of $S(V/E)_\odot$. For any base point $p\in S(V/E)_\odot$, $\pi_1(S(V/E)_\odot,p)$ gives a parabolic subgroup $P$ of $G$ (cf. Definition~\ref{d:parabolic}) and the Hopf fiber passing through $p$ represents the central $\mathbb Z$-subgroup $Z_P$.

\begin{propC}\label{propC}
	For each $E\in \zci$ and a basepoint $p\in \db_E X$, the image of $\pi_1(\db_E X,p)$ in $\pi_1(X,p)=G$ is equal to the normalizer of $Z_E$ in $G$ where $Z_E$ is represented by the Hopf fiber of $S(V/E)_\odot$ passing through $p$.
\end{propC}

Before we formulate the next property, we record the following lemma, which is a straightforward consequence of the relevant definitions.

\begin{lemma}
	\label{l:commute}
	Let $\alpha$ be a simplex of $\ci_0$. Take $p\in \partial_\alpha X$. For each $E\in\vertex\alpha$, consider the Hopf fiber for each $S(V/E)_\odot$ passing through $p$. This gives rise to a collection of mutually commuting $\mathbb Z$-subgroups in $G$ (and hence, a free abelian subgroup of rank $=\dim \ga +1$). 
\end{lemma}

\begin{propD}\label{propD}
	For $k\ge 2$ and for any collection $\{L_1,L_2,\cdots,L_k\}$ of $k$ commuting $\zz$-subgroups such that each of $L_i$ is the central $\mathbb Z$- subgroup of some parabolic subgroup of $G$, there exists $g\in G$ so that $g\{L_1,L_2,\cdots,L_k\} g^{-1}$ arises from a simplex of $\ci_0$ as described in Lemma~\ref{l:commute}.
\end{propD}

\begin{proposition}
	\label{prop:iso}
	Suppose $\ca$ is an essential complex arrangement in an $n$-dimensional complex vector space. Suppose $\ca$ satisfies properties A, B, C and D. Then the following statements are true.
	\begin{enumerate1}
		\item $\ctop$ is naturally isomorphic to $\calg$.
		\item The action of fundamental group on $\mathcal{C}=\ctop=\calg$ has no inversions (i.e., the pointwise stabilizer of each simplex equal to its setwise stabilizer), and the quotient complex is naturally isomorphic to the complex $\ci_0$ of irreducible subspaces.
		\item If $\db_\alpha X$ is aspherical for each simplex $\alpha$ of $\ci_0$, then $\mathcal{C}$ has the same homology as a wedge of spheres of dimension $n-k-1$, where $k$ is the number of irreducible components of $\ca$. 
		\item If $\dim\mathcal{C}\le 1$, then  $\mathcal{C}$ is homotopy equivalent to a wedge of spheres. If $\dim \mathcal{C}>1$ and if $\mathcal{C}$ is simply connected, then $\mathcal{C}$ is homotopic to a wedge of spheres.
	\end{enumerate1}
\end{proposition}

\begin{proof}
	First we define a map $f:\ctop^{(0)}\to \calg^{(0)}$. Each vertex $v$ of $\ctop^{(0)}$ corresponds to a face $\db_v\tilde X$ with a product decomposition $$\db_v\tilde X\cong \widetilde{SE_\odot}\times\widetilde{S(V/E)_\odot}$$ where $E$ is the type of $v$. The setwise stabilizer of each $\widetilde{S(V/E)_\odot}$ fiber is the same, namely, a parabolic subgroup of $G$ and $f(v)$ is defined to be the center of this parabolic subgroup. As this parabolic subgroup can not correspond to an irreducible factor of $\ca$, we know that $f(v)$ gives a vertex in $\calg^{(0)}$. Note that $f$ is surjective, $G$-equivariant and type preserving. Suppose for vertices $v_1$ and $v_2$ we have $f(v_1)=f(v_2)$. Then $v_1$ and $v_2$ have the same type. Thus there exists $g\in G$ such that $g\db_{v_1}\tilde X=\db_{v_2}\tilde X$. Suppose $\mathbb Z_{v_1}$ is the $\mathbb Z$-subgroup associated with $f(v_1)$. As $f$ is $G$-equivariant, $g$ normalizes $\mathbb Z_{v_1}$. By Property C, $g$ stabilizes $\db_{v_1}\tilde X$. Thus, $v_1=v_2$ and $f$ is injective. For two adjacent vertices $v_1$ and $v_2$ in $\ctop$, we have $\db_{v_1}\tilde X\cap \db_{v_2}\tilde X\neq\emptyset$. So, $\db_{E_1}X\cap \db_{E_2}X\neq\emptyset$ where $E_i$ is the type of $v_i$. By Theorem~\ref{thm:higher codimension pieces}, $\{E_1,E_2\}$ is a nested set. By Lemma~\ref{l:commute}, $f(v_1)$ and $f(v_2)$ commute. So, if a collection of vertices of $\ctop$ spans a simplex, then their $f$-images span a simplex. By Property D, the converse of this statement also holds. So to show that $f$ extends to an isomorphism from $\ctop$ to $\calg$, it suffices to prove that $\ctop$ is a simplicial complex. Equivalently, we need to prove if $\{v_i\}_{i=1}^k$ is a collection of vertices of $\ctop$ spanning a simplex, then $\bigcap_{i=1}^k\partial_{v_i}X$ is connected. The case $k=2$ follows from Properties $A'$ and $B$ and Lemma~\ref{l:conintersect} below. We can use Property B to reduce the general case to the case $k=2$ by noticing that $\bigcap_{i=1}^k\partial_{v_i}X=(\bigcap_{i=1}^{k-1}\partial_{v_i}X)\cap(\bigcap_{i=2}^k\partial_{v_i}X)$. This finishes the proof of the first assertion.
	
	For (2), as the action of $G$ on $\calg$ is type preserving, and different vertices in a simplex have different types, the action does not have inversions. As $\ctop=\Delta(\tilde X)$ and the $(k-1)$-simplices of $\Delta(\tilde X)$ are in one-to-one correspondence with codimension-$f$ faces of $\tilde X$, we know that the $G$-orbits of simplices in $\ctop$ are in one-to-one correspondence with the simplices of $\ci_0$. This proves (2).
	
	In order to prove (3) we begin with the fact  that $\cm(\ca)$ is homotopy equivalent to a CW complex of dimension $n$.  It is proved in \cite{djlo} that for any $\ca$ that is essential and central, $H^*(\cm(\ca);\zz G)$ is  concentrated in degree $n$ and it is a free abelian group in that degree.  (This is a generalization a result of Squier \cite{squier}, where this is proved in the case of the complexification of a reflection arrangement.)   Since $\cm(\ca)$ is homotopy equivalent to the compact space $X$, we have $H^*(X;\zz G)\cong H^*_c(\ti{X})$.  Since each face of $\partial\tilde X$ is contractible and since $\cac$ is the nerve of covering of $\db\ti{X}$ by its codimension-one faces,  $\db\ti{X}$ is homotopy equivalent to $\cac$. Note that $\dim \cac=\dim \ci_0= n-k-1$ since $\ci_0$ is the join of $k$ factors of the form $\ci_0(\ca_i)$. Since $(\ti{X},\db\ti{X})$ is a manifold with boundary of dimension $2n-k$, Poincar\'e duality gives:
	\begin{align*} 
	H^*_c(\ti{X}) & \cong H_{2n-k -*}(\ti{X},\db\ti{X}) \notag \\
	& \cong \ti{H}_{2n-k-*-1}(\db\ti{X})\cong \ti{H}_{2n-k-*-1}(\cac)\, \notag \\
	\end{align*}
	where the second  equation holds since $\ti{X}$ is contractible.  Since $H^*_c(\ti{X})$ is concentrated in degree $n$, $H_{*}(\ti{X},\db\ti{X})$ is concentrated in degree $n-k$ and hence, $\ti{H}_*(\cac)$ is concentrated in degree $n-k-1$ ($=\dim \cac$).  This proves (3). (4) follows immediately from (3).
\end{proof}

\begin{remark}
	Suppose $\ca$ satisfy property A, B, C and D. When $\dim \ctop>1$, $\cac$ being simply-connected is equivalent to that the map $\partial X\to X$ induces an isomorphism on the fundamental groups. In fact, we conjecture that this is true for any central arrangement $\ca$.
\end{remark}

\begin{lemma}\label{l:conintersect}
	Suppose $B=B_0\cup B_1$, where $B$, $B_0$, $B_1$ and $C=B_0\cap B_1$ are path connected. Take base point $x_0\in C$. Let $h:\pi_1(B,x_0)\to G$ be a homomorphism to some group $G$ such that
	\begin{enumerate1}
		\item for $i=0,1$, $\pi_1(B_i, x_0)\to \pi_1(B,x_0)\to G$ is injective (denote the image of $\pi_1(B_i,x_0)$ in $G$ by $G_i$) and $\pi_1(C, x_0)\to \pi_1(B,x_0)\to G$ is injective (denote the image in $G$ by $H$),
		\item $G_0\cap G_1=H$.
	\end{enumerate1}
	Let $\pi:\tB\to B$ be a regular covering space with group of deck transformations $\image h$. Then if $C'$ is a component of $\pi^\minus (C)$ and if
	$B_i'$ is a component of $\pi^\minus(B_i)$ that contains $C'$, then $B'_0\cap B'_1= C'$. 
\end{lemma}

\begin{proof}
	Assume $G=\image h$ for simplicity. Suppose $B'_0\cap B'_1$ contain a connected component $C''$ which is different from $C'$. Let $x'\in C'$ and $x''\in C''$ be lift of $x_0\in C$. For $i=1,2$, let $\omega_i$ be a path from $x'$ to $x''$ inside $B'_i$. Let $g_i\in H_i$ be the element represented by $\pi(\omega_i)$. As $\omega_0\omega^{-1}_1$ is a loop in $\tilde B$, $g_0g^{-1}$ represents the trivial element in $G$. So, $g_0=g_1$. By condition (2), $g_0\in H$, which contradicts that $C'\neq C''$. The lemma follows.
\end{proof}

\paragraph{Admissible families of arrangements}

\begin{definition}
	\label{def:admissible}
	Let $\mathfrak C$ be a collection of complex hyperplane arrangements. We say $\mathfrak C$ is \emph{admissible} if the following conditions hold.
	\begin{enumerate1}
		\item if $\ca\in\mathfrak C$ is an arrangement in vector space $V$ and $E\in\cq(\ca)$, then $\ca_{V/E}\in \mathfrak C$ and $\ca^E\in\mathfrak C$; 
		\item if $\ca_1,\ca_2\in\mathfrak C$, then $\ca_1\oplus \ca_2\in\mathfrak C$.
	\end{enumerate1}
\end{definition}

For example, the collection of all complex hyperplane arrangements that are complexifications of simplicial real arrangements is admissible \cite{deligne}. The collection of supersolvable arrangements is also admissible (see \cite[Section 2]{supersolvable} for a summary).

\begin{proposition}
	\label{prop:induction}
	Suppose $\mathfrak C$ is an admissible collection of complex hyperplane arrangements. Suppose each element in $\mathfrak C$ satisfies Properties A, C and D. Then the following statements are true.
	\begin{enumerate1}
		\item Suppose $\ca\in\mathfrak C$. Then the stabilizer of each simplex in $\calg$ is isomorphic to $\pi_1(\cm(\ca'))$ for some $\ca'\in\mathfrak C$.
		\item Each element of $\mathfrak C$ satisfies property A, B, C and D.
	\end{enumerate1}	
\end{proposition}

\begin{proof}
	We prove (1). Suppose the simplex has vertex set $\{v_i\}_{i=1}^k$. Let $\zz_i$ be the central $\zz$-subgroup of the parabolic subgroup associated with $v_i$ of type $E_i$.  By Property D, $\{E_i\}_{i=1}^k$ is a nested subset. Let $H_i$ be the stabilizer of $v_i$. Then $H_i=N_G(\zz_i)=C_G(\zz_i)$ by property C. Suppose $E_1$ does not properly contain any other $E_i$. There is a natural map $f:G_1:=\pi_1(\cm(\ca_{V/E_1}\oplus \ca^{E_1}))\to H_1$, which is an isomorphism by Properties A and C. Note that $\bigcap_{i=1}^kC_G(\zz_i)=\bigcap_{i=2}^kC_{H_1}(\zz_i)$. It follows from Property D that if $E_1\subset E_i$, then $f^{-1}(\zz_i)$ is a central $\zz$-subgroup of a parabolic subgroup in $\pi_1(\cm(\ca_{V/E_1}))$ that is irreducible. The case $\ca_{V/(E_1\cap E_i)}\cong \ca_{V/E_1}\oplus \ca_{V/E_i}$ is similar, where $f^{-1}(\zz_i)$ is a central $\zz$-subgroup of an irreducible parabolic subgroup in $\pi_1(\cm(\ca^{E_1}))$. Thus, we are reduced to the stabilizer of a simplex with $k-1$ vertices in the curve complex for $\ca_{V/E_1}\oplus \ca^{E_1}$. The arrangement $\ca_{V/E_1}\oplus \ca^{E_1}$ is in $\mathfrak C$. So we are done by induction. Next we prove (2). We use the same notation as before. By Lemma~\ref{lem:inductive} and Properties A, C for $\ca_{V/E_1}\oplus \ca^{E_1}$, the inclusion $\partial_{E_1} X\cap\partial_{E_i} X\to \partial_{E_1} X$ induces injective map on fundamental groups and its $\pi_1$-image is $C_{H_1}(\zz_i)$. The case $k=2$ follows. The more general case can be proved by induction as in (1). 
\end{proof}

\section{Zonotopes, Salvetti complexes and Deligne groupoids}
\label{sec:background}
This section deals with some complexes associated to a real, simplicial arrangement of hyperplanes.  So, $\ca$ denotes an essential, central arrangement of hyperplanes in a real vector space $A$. 

\subsection{Real arrangements and their dual zonotopes}\label{ss:zonotope}
The hyperplanes in $\ca$ cut $A$ into a collection of convex polyhedral cones, called the \emph{fan} of $\ca$, denoted by $\fan(\ca)$.
The intersection of these polyhedral cones with the sphere $SA$ defines  a cellulation of $SA$ by totally geodesic spherical polytopes. This gives rise $SA$ to a piecewise linear cellulation of $SA$, and we denote the associated cell complex by $\partial\fan(\ca)$. We define the \emph{dual zonotope} of $\ca$, denoted $Z(\ca)$, to be $SA$  with the cell structure dual to $\partial\fan(\ca)$. The real arrangement is \emph{simplicial} if $\partial\fan(\ca)$ is a simplicial complex.

\begin{example}\label{ex:permutohedron} If $\ca$ is the braid arrangement corresponding to the action of the symmetric group $S_{n+1}$ on $A=\rr^{n}$, then $\db\fan(\ca)$ is the Coxeter complex and $Z(\ca)$ is the permutohedron.
\end{example}

We denote the barycentric subdivision of $Z(\ca)$ by $bZ(\ca)$, which is a simplicial complex (the  ``$b$'' indicates ``barycentric subdivision''). Note that $bZ(\ca)$ and $b\partial\fan(\ca)$ can be naturally identified. One can realize $Z(\ca)$ as a convex polytope in the dual space $A^*$ of $A$. However, in this paper, we would like to embed $Z(\ca)$ (and $bZ(\ca)$) as a (not necessarily convex) piecewise linear subset of $A$ as follows.

For each cone $U\in \fan(\ca)$, choose a point $x_U$ in the relative interior of $U$. The partial order on $\fan(\ca)$ is defined by $U_1<U_2$ if $U_1$ is contained in the closure of $U_2$ and in this case, we also write $x_{U_1}<x_{U_2}$. Each chain $x_{U_1}<x_{U_2}<\cdots<x_{U_k}$,  determines a simplex in $A$ with vertex set $\{x_{U_i}\}_{i=1}^k$. This defines an embedding of the simplicial complex $bZ(\ca)$ as a subset of  $A$. 
The simplices of  $bZ(\ca)$ can be assembled into cells in two different ways to obtain the cell structure on $\partial\fan(\ca)$ and $Z(\ca)$.  
The faces of the zonotope $Z(\ca)$ are in one-to-one correspondence with vertices of $bZ(\ca)$.  We identify the face of $Z(\ca)$ associated with vertex $x_U\in bZ(\ca)$ with the union of all simplices of $bZ(\ca)$ corresponding to chains whose smallest element is $x_U$.  In this way each vertex of $bZ(\ca)$ can also be regarded as the barycenter of a face of $Z(\ca)$.

Suppose $B$ is a subspace in $\cq(\ca)$. A face $F$ of $Z(\ca)$ is \emph{dual} to $B$ if $F\cap B=\{b_F\}$, where $b_F$ denotes the barycenter of $F$. It follows that $F$ is isomorphic to the zonotope dual to the real arrangement $\ca_{A/B}$. Note that $\fan(\ca^B)\subset\fan(\ca)$ So, our previous choices of the $x_{U_i}$'s  yields an embedding $bZ(\ca^B)\to B$. Therefore, we can treat $bZ(\ca^B)$ as a subcomplex of $bZ(\ca)$, i.e., $bZ(\ca^B)=bZ(\ca)\cap B$.

Two faces $F_1$ and $F_2$ of $Z(\ca)$ are \emph{parallel} if they are dual to the same subspace in $\cq(\ca)$: we write $F_1\parallel F_2$. When $Z(\ca)$ is realized as a convex polytope in $A^*$ dual to $\fan(\ca)$, then $F_1$ and $F_2$ actually are parallel faces of $Z(\ca)$. Parallel classes of faces of $Z(\ca)$ are in one-to-one correspondence with subspaces of $\cq(\ca)$.  For example, the duals to a hyperplane in $\ca$ form a parallel class of edges in $Z(\ca)$. 

Given  two faces $F_1$ and $F_2$ of $Z(\ca)$ both dual to $B\in \cq(\ca)$,  define $p:\vertex F_1\to \vertex F_2$ as follows. Let $\ca_B$ be the collection of hyperplanes of $\ca$ containing $B$. For $x\in\vertex F_1$, define $p(x)$ to be the unique vertex in $F_2$ such that $x$ and $p(x)$ are not separated by any hyperplanes in $\ca_B$. The map $p$ is called \emph{parallel translation}.  (When $F_1$ and $F_2$ are regarded as actual parallel faces of the zonotope $Z(\ca)$,  the map $p$ is the restriction of the parallel translation taking $F_1$ to $F_2$.)

Define $F_1$ to be \emph{orthogonal} to $F_2$, denoted by $F_1\perp F_2$, if there is another face $F$ of $Z(\ca)$ containing $F_1$ and $F_2$ such that $F\cong F_1\times F_2$ (and consequently, $F_1\cap F_2$ will be a vertex). Note that $F_1\perp F_2$ if and only if $\ca_{A/B_1}\times \ca_{A/B_2}=\ca_{A/(B_1\cap B_2)}$ where $B_i$ is the subspace dual to $F_i$.

The 1-skeleton of $Z(\ca)$ is endowed with a path metric $d$ such that each edge has length 1. Given $x,y\in \vertex Z(\ca)$ it turns out that $d(x,y)$  is the number of hyperplanes separating $x$ and $y$ (cf.\ \cite[Lemma 1.3]{deligne}).  In the next lemma 	we collect some basic facts which will be used in sections~\ref{sec:gar} and \ref{sec:sal}. 

\begin{lemma}
	\label{lem:gate}
	Let $x$ be a vertex in $Z(\ca)$ and $F$ be a face of $Z(\ca)$.
	\begin{enumerate1}
		\item There exists a unique vertex $x_F\in F$ such that $d(x,x_F)\le d(x,y)$ for any vertex $y\in F$. The vertex $x_F$ is called the \emph{projection} of $x$ to $F$, and is denoted $\prj_F(x)$. 
		\item For any vertex $y\in F$, there exists a shortest edge path $\omega$ in the 1-skeleton of $Z(\ca)$ from $x$ to $y$ so that $\omega$ passes through $x_F$ and so that  the segment of $\omega$ between $x_F$ and $y$ is contained in $F$.
		\item Let $\ch_F$ be the collection of hyperplanes in $\ca$ dual to some edge of $F$. Then $x_F$ can be characterized as the unique vertex in $F$ such that	no element in $\ch_F$ separates $x$ from $x_F$.
		\item Let $F'$ be another face parallel to $F$ and let $p:F'\to F$ be parallel translation. Let $y\in F'$. Then  $\prj_F(y)=p(y)$.
	\end{enumerate1}
\end{lemma}
Statements (1) and (2) are proved  in \cite[Lemma 3]{s87}. Statements (3) and (4) follow immediately from (1) and(2).

\begin{lemma}\label{lem:span}
	Suppose $\ca$ is finite, central, simplicial arrangement. Then the following statements are true.
	\begin{enumerate1}
		\item Given edges $e_1, \dots, e_k$ of $Z(\ca)$ sharing a vertex $x$, let $F$ be the smallest face of $Z(\ca)$ containing each $e_i$. 		Then any edge of $F$ containing $x$ has to be one of the  $e_i$. This $F$ is called the face \emph{spanned} by $\{e_i\}_{i=1}^k$.
		\item Let $F\subset Z(\ca)$ be a face containing  a given vertex $x$. Let $\{e_1,\ldots,e_l\}\subset Z(\ca)$ be a set of distinct edges that contain $x$ but do not lie in $F$. Then there is a unique face $F'\subset Z(\ca)$ containing $F$ and $\{e_1,\ldots,e_l\}$ such that $\dim(F')=\dim(F)+l$. This $F'$ is called the face \emph{spanned} by $F$ and $\{e_1,\ldots,e_l\}$.
	\end{enumerate1}
\end{lemma}
This lemma follows immediately from the fact that the link of each vertex of $Z(\ca)$ is a simplex (since $\ca$ is simplicial). Also, with regard to assertion (2), suppose $\{e'_i\}_{i=1}^k$ is the set of  edges of $F$ that contain $x$. Then $F'$ is the face spanned by $\{e'_1,\ldots,e'_k,e_1,\ldots,e_l\}$.

\begin{remark}
	\label{rmk:inherit}
	Suppose $\ca$ is finite, central and simplicial. Let $E\in \cq(\ca)$ be a subspace. Then both $\ca^E$ and $\ca_{V/E}$ are simplicial, c.f.\ \cite[Lemma 2.17]{supersolvable}.
\end{remark}

\subsection{The Salvetti complex}\label{ss:sal} 
Suppose $\ca$ is  a real arrangement in a vector space $A$ ($\cong\rr^n$) and that $Z$ ($=Z(\ca)$) is the dual zonotope. Consider the set of pairs $(F,v)\in \cp(Z)\times \vertex Z$.  Define  an equivalence relation $\sim$ on this set by $$(F,v)\sim (F,v') \iff F=F' \text{\ and\ } \prj_F(v') = \prj_F(v).$$
Denote the equivalence class of $(F,v')$ by $[F,v']$ and let $\ce(\ca)$ be  the set of equivalence classes.   Note  that each equivalence class $[F,v']$ contains a unique representative of the form $(F,v)$, with $v\in \vertex F$.   The partial order on $\cp(Z)$ induces a partial order on $\ce(\ca)$.  In  \cite{s87} the \emph{Salvetti complex} $\s(\ca)$ of $\ca$ is defined as the regular CW complex given by taking  $Z\times \vertex Z$ (i.e., a disjoint union of copies of $Z$) and then identifying faces $F\times v$ and $F\times v'$ whenever $[F,v]=[F,v']$, i.e.,
\begin{equation}\label{e:sal}
\s(\ca)=( Z\times \vertex Z)/ \sim \ .
\end{equation}
For example, for each edge $F\in \edge Z$ $Z^{(1)}$  with endpoints $v_0$ and $v_1$, we get two $1$-cells $[F,v_0]$ and $[F,v_1]$ of $\s(\ca)$ glued together along their endpoints $[v_0,v_0]$ and $[v_1,v_1]$.  So, the $0$-skeleton of $\s(\ca)$ is equal to the $0$-skeleton of $Z$ while its $1$-skeleton is formed from the $1$-skeleton of $Z$ by doubling each edge.  

Since $\s(\ca)$ is a union of cells of the form $(Z(\ca),v)$ with $v$ ranging over vertices of $Z(\ca)$, there is a natural map $\s(\ca)\to Z(\ca)$ defined by ignoring the second coordinate. A \emph{standard subcomplex} of $\s(\ca)$ is the inverse image of a face of $Z(\ca)$ under this map. A standard subcomplex is \emph{proper} if it is associated with a positive dimensional proper face of $Z(\ca)$, i.e., a face which is neither a vertex nor the entire zonotope $Z(\ca)$.

\begin{remark}\label{rmk:orientation}
	Each edge of $\s(\ca)$ has a natural orientation, namely, if $F=\{v_0,v_1\}$ is an edge of $Z$, then $[F,v_0]$ is oriented so that $[v_0,v_0]$ is its initial vertex and $[v_1,v_1]$ is its terminal vertex.  An edge path in the $\s(\ca)$ is \emph{positive} if each of its edges is positively oriented. (Positive paths are related to the Deligne groupoid defined in Section~\ref{subsec:deligne groupoid} below.)
\end{remark}

For example, if $Z(\ca)$ is a hexagon, then $\s(\ca)$ is a CW complex with six vertices, twelve edges and six 2-cells. 

There is a natural projection $\pi:\s(\ca)\to Z(\ca)$.

The main property of $\s(\ca)$ is that it has the homotopy type of the complement of the complexified arrangement $\ca\otimes \cc$.

We write a simplex in the barycentric subdivision $b\s(\ca)$ of $\s(\ca)$ as a pair $(\Delta,v)$, where $\Delta$ is a simplex of $bZ(\ca)$ and $v$ is a vertex of $Z(\ca)$. 

\begin{remark}\label{rmk:intersection}
	Suppose $\Delta=[b_{F_0},b_{F_1},\ldots,b_{F_k}]$ where each $F_i$ is a face of $Z(\ca)$ and $b_{F_i}$ is the barycenter of $F_i$. We assume $F_k\subset F_{k-1}\subset\cdots\subset F_0$. 
	Let $v_1$ and $v_2$ be two vertices of $Z(\ca)$. 
	Note that if $\prj_{F_i}(v_1)=\prj_{F_i}(v_2)$, then $\prj_{F_j}(v_1)=\prj_{F_j}(v_2)$ for any $j<i$. Then $(\Delta,v_1)\cap (\Delta,v_2)=(\Delta',v)$, where $\Delta'=[b_{F_0},b_{F_1},\ldots,b_{F_m}]$ where $m\le k$ is the largest possible number such that $\prj_{F_m}(v_1)=\prj_{F_m}(v_2)$ and $v=\prj_{F_m}(v_2)$. (If $\prj_{F_0}(v_1)\neq \prj_{F_0}(v_2)$, then $(\Delta,v_1)\cap (\Delta,v_2)=\emptyset$.) 
\end{remark}

\subsection{The Deligne groupoid}
\label{subsec:deligne groupoid}
We  recall the definition of the ``Deligne groupoid'' from \cite{deligne} and \cite{Paris00}. A \emph{chamber} of $\ca$ is a connected component of $A-(\bigcup_{H\in \ca}H)$. 
In other words, a chamber is the interior of a top-dimensiponal cone in $\fan(\ca)$. 
Two chambers $C$ and $D$ are \emph{adjacent} if there exists exactly one hyperplane in $\ca$ separating $C$ from $D$. The chambers of $\ca$ are in bijective correspondence with the vertices of $Z(\ca)$ and two chambers are adjacent if there is an edge of $Z(\ca)$ connecting the corresponding vertices.) Let $\Ga(\ca)$ be the directed graph whose vertices are the chambers, and whose arrows are pairs $(C,D)$ of adjacent chambers ($(C,D)$ and $(D,C)$ are distinct oriented edges). Then $\Ga(\ca)$ can be identified with the 1-skeleton of the Salvetti complex $\s(\ca)$ with edge orientations as in Remark~\ref{rmk:orientation}. We identify the vertices of $\Ga(\ca)$ with the vertices of $Z(\ca)$.  (A vertex of $\Ga(\ca)$ is identified with the vertex $[v,v]$ of $\s(\ca)$, and then with the vertex $v$ of  $Z(\ca)$.) So, there is a natural map $\pi:\Ga(\ca)\to Z(\ca)$ whose image is the 1-skeleton of $Z(\ca)$. (We think of $\Ga(\ca)$ as the ``doubled $1$-skeleton of $Z(\ca)$''). Given an edge $e$ of $\Ga(\ca)$, write $\bar e$ for $\pi(e)$.

Let $E(\Gamma)$ be the collection of edges of $\Ga(\ca)$. Since the edges of $\Ga(\ca)$ are directed, each element $a\in E(\Gamma)$ has a \emph{source}, denoted $s(a)$, and a \emph{target}, denoted $t(a)$. Introduce a formal inverse of $a$, denoted $a^{-1}$ It can be thought as traveling  the same edge but in the opposite direction. Thus, $t(a^{-1})=s(a)$ and $s(a^{-1})=t(a)$. 

A \emph{path} of $\Gamma$ is an expression $g=a^{\eps_1}_1a^{\eps_2}_2\cdots a^{\eps_n}_n$ where $a_i\in E(\Gamma)$, $\eps_i\in\{\pm 1\}$ and $t(a^{\eps_i}_i)=s(a^{\eps_{i+1}}_{i+1})$. Define $s(g)=s(a^{\eps_1}_1)$ and $t(g)=t(a^{\eps_n}_n)$. The \emph{length} of $g$ is $n$. A vertex is a path of length $0$. The path $g$ is \emph{positive} if $\eps_i=1$ for each $1\le i\le n$. The path $g$ is \emph{minimal} if any path from $s(g)$ to $t(g)$ has length $\ge n$. 

Let $\sim$ be the smallest equivalence relation on the set of paths such that
\begin{enumeratea}
	\item $ff^\minus \sim s(f)$ for any path $f$;
	\item if $f\sim g$, then $f^{-1}\sim g^{-1}$;
	\item if $f\sim g$, and $h_1$ is a path with $t(h_1)=s(f)=s(g)$, and $h_2$ is a path with $s(h_2)=t(g)=t(g)$, then $h_1fh_2\sim h_1gh_2$;
	\item if $f$ and $g$ are both minimal positive paths with $s(f)=s(g)$ and $t(f)=t(g)$, then $f\sim g$.
\end{enumeratea} 

Let $[f]$ be the collection of all paths equivalent to $f$. Define another equivalence relation on the set of positive paths, called \emph{positive equivalence} and denoted $\sim_+$: it is the smallest equivalence relation generated by conditions (c) and (d) above. If $f$ is positive, let $[f]_+$ be the set of all positive paths that are positively equivalent to $f$. Note that the elements of $[f]_+$ all have the same length; so, $[f]_+$ has a well-defined length.

Define $\g(\ca)$ (resp. $\g^+(\ca)$) to be the collection of all equivalence (resp. positive equivalence) classes of paths (resp. positive paths). The category $\g(\ca)$ (resp. $\g^+(\ca)$) has the structure of a groupoid (resp. a category), whose objects are vertices of $\Ga(\ca)$, whose morphisms are equivalence classes of paths (resp. positive paths) with compositions given by concatenation of paths.  Then $\g(\ca)$ is called the \emph{Deligne groupoid}.  Put $\g=\g(\ca)$ and $\g^+=\g^+(\ca)$. For objects $x,y\in\g^+$, let $\g^+_{x\to}$ denote the collection of morphisms whose source  is $x$. Define $\g^+_{\to y}$ and $\g^+_{x\to y}$ similarly. These notions  can be defined for $\g$ (that is, without using the superscript ${}^+$) in the same way. Note that the collection of morphisms of $\g$ with source and target both equal to $x$ is a group, called the \emph{isotropy group} at $x$  denoted by $\g_x$.

For two morphisms (i.e., for two positive equivalence classes of positive paths) $f$ and $g$ in $\g^+$, 
define the \emph{prefix order} in such a way that $f\preccurlyeq g$ if there is a morphism $h$ such that $g=fh$. Similarly, define the \emph{suffix order} so that $g\succcurlyeq f$ if there is a morphism $h$ such that $g=hf$. Then $(\g^+_{x\to},\preccurlyeq)$ and $(\g^+_{\to y},\succcurlyeq)$ are posets. 

\begin{Theorem}\label{thm:deligne} 	\textup{(\cite{deligne}).}
	Suppose the real hyperplane arrangement $\ca$ is a finite, central and simplicial . The following statements are true.
	\begin{enumerate1}
		\item The natural map $\g^+\to\g$ is injective.  In other words, two positive path are positively equivalent if and only if they are equivalent.  Moreover, left and right cancellation laws both hold in $\g^+$.
		\item For any vertex $x\in\Ga(\ca)$, the posets $(\g^+_{x\to},\preccurlyeq)$ and $(\g^+_{\to x},\succcurlyeq)$ are lattices.
		\item For any vertex $x\in\Ga(\ca)$, the posets $(\g_{x\to},\preccurlyeq)$ and $(\g_{\to x},\succcurlyeq)$ are lattices.
	\end{enumerate1}
\end{Theorem}
For two morphisms $f$ and $g$ with $s(f)=s(g)$ (resp. $t(f)=t(g)$), we write $f\vee_p g$ and $f\wedge_p g$ (resp. $f\vee_s g$ and $f\wedge_s g$) for the join and meet of $f$ and $g$ with respect to the prefix order (resp. suffix order).  (Here, as usual, the words ``join'' and``meet'' mean ``least upper bound'' and `` greatest lower bound,'' respectively).

For paths $f$ and $g$, write $f\eq g$  if $[f]=[g]$.

\begin{lemma}
	\label{lem:pn form}
	Given any path $f$ on $\Ga(\ca)$, there exist positive paths $a$ and $b$ such that $f\eq ab^{-1}$ and $a\wedge_s b=t(a)$. Moreover, if $f\eq cd^{-1}$ where $c$ and $d$ are positive paths with $c\wedge_s d=t(c)$, then $a\eq c$ and $b\eq d$. Thus if $f\eq a_1b^{-1}_1$ for $a_1$ and $b_1$ positive, then $a\pr a_1$ and $b\pr b_1$.
\end{lemma}
In the above lemma $t(a)$  denotes the identity morphism at the vertex $t(a)$. The proof of Lemma~\ref{lem:pn form} is identical to that of \cite[Theorem 2.6]{charney}. The decomposition $f\eq ab^{-1}$ is called the \emph{$pn$-normal form} of $f$.

Given a path $f=a^{\eps_1}_1a^{\eps_2}_2\cdots a^{\eps_n}_n$,  define the \emph{signed intersection number}, denoted $i(f,H)$, to be the sum of all the $\eps_i$'s such that $a_i$ is dual to $H$. In the special case when $f$ is positive, $i(f,H)$ is the number of times the edge path $\pi(f)$ crosses $H$. Two positive minimal paths with the same end points cross the same collection of hyperplanes (and each hyperplane is crossed exactly once).  This gives the following lemma.

\begin{lemma}
	\label{lem:intersection number}
	Let $f$ and $g$ be paths on $\Ga(\ca)$ such that $f\sim g$. Then $i(f,H)=i(g,H)$ for any $H\in \ca$.
\end{lemma}

\subsection{Some facts about irreducible arrangements}
Let $Z(\ca)$ be the dual zonotope of $\ca$ and $x$ a vertex in $Z(\ca)$. Following \cite{supersolvable}, the \emph{Coxeter graph} of $\ca$ at $x$, denoted by $\Gamma_x$, is defined as follows. 
Let $\{e_i\}_{i=1}^k$ be the collection of edges of $Z(\ca)$ containing $x$, and let $H_i$ be the hyperplane of $\ca$ dual to $e_i$. Let $B_{ij}=H_i\cap H_j$. The vertices of $\Gamma_x$ are in one-to-one correspondence with $\{e_i\}_{i=1}^k$. Two distinct vertices are joined by an edge if $\ca_{A/B_{ij}}$ is irreducible (i.e., if  it contains more than two lines).

\begin{lemma}\textup{(\cite{supersolvable}*{Lemma 3.5}).}
	\label{lem:irreducible characterization}
	Let $\ca$ and $Z(\ca)$ be as above. Then the following statements are equivalent.
	\begin{enumerate1}
		\item The arrangement $\ca$ is an irreducible.
		\item The Coxeter graph $\Gamma_x$ is connected for some vertex $x\in Z(\ca)$.
		\item The Coxeter graph $\Gamma_x$ is connected for any vertex $x\in Z(\ca)$.
	\end{enumerate1}
\end{lemma}
A face $F$ of $Z(\ca)$ that is dual to a subspace $B\in \cq(\ca)$ is \emph{irreducible} if $\ca_{A/B}$ is irreducible. Then we have the following corollary to Lemma~\ref{lem:irreducible characterization}.

\begin{corollary}
	\label{cor:irreducible}
	Let $x\in Z(\ca)$ be a vertex. Suppose there exists a pair of irreducible faces $F_1$ and $F_2$ of $Z(\ca)$ such that
	\begin{enumerate1}
		\item any edge of $Z(\ca)$ containing $x$ is contained in either $F_1$ or $F_2$;
		\item there exists an edge $e$ such that $x\in e$ and $e\subset F_1\cap F_2$.
	\end{enumerate1}
	Then $\ca$ is irreducible.
\end{corollary}

\begin{lemma}\textup{(\cite{supersolvable}*{Lemma 3.11}).}
	\label{lem:restriction irreducible}
	Suppose $\ca$ is a finite, central, simplicial real arrangement. If $\ca$ is irreducible, then $\ca^B$ is irreducible and simplicial for any subspace $B\in\cq(\ca)$.
\end{lemma}

\section{Garside theoretic computations}
\label{sec:gar}
In this section $\ca$ is, as before,  an essential, simplicial arrangement of linear hyperplanes in a real vector space $A$. Let $V=A\otimes \mathbb C$ and let $\ca\otimes \mathbb C$ be the complexification of $\ca$. 
\subsection{Faces of zonotopes and words of longest length}\label{subsec:face}

Let $Z(\ca)$ be the zonotope dual to $\ca$ and let $\Ga(\ca)$, $\g^+=\g^+(\ca)$, $\g=\g(\ca)$, and $\pi:\Ga(\ca)\to Z(\ca)$ be as in Subsection~\ref{subsec:deligne groupoid}. 
\begin{definition}
	Suppose $F$ is a face of $Z(\ca)$.  An \emph{$F$-path} of $\Ga(\ca)$ is a path whose image under $\pi$ is contained in $F$.
\end{definition}

Put $\Ga(\ca,F)=\pi^{-1}(F)$. We can identify $\Ga(\ca,F)$ with $\Ga(\ca_{A/B})$, where $B$ is the subspace dual to $F$. Let $\G(F)$ denote the Deligne groupoid  over $\Ga(\ca,F)$ so that $\G(F)$ is isomorphic to $\G(\ca_{A/B})$. By Remark~\ref{rmk:inherit}, $\G(F)$ also satisfies Theorem~\ref{thm:deligne}.

Define a retraction $r_F:\Ga(\ca)\to \Ga(\ca,F)$ as follows. Given a vertex $x\in \Ga(\ca)$, put $r_F(x)=\prj_F(x)$, where $\prj_F$ is the ``nearest vertex projection'' defined in Lemma~\ref{lem:gate}. 
Let $e$ be an oriented edge $\Ga(\ca)$. Suppose $H_e\in\cq(\ca)$ denotes the hyperplane dual to $e$. If $H_e$ is not dual to any edge of $F$, then $r_F(s(e))=r_F(t(e))$ and we define $r_F(e):=r_F(s(e))$. If $H_e$ is dual to an edge of $F$, then $r_F(e)$ is the oriented edge from $r_F(s(e))$ to $r_F(t(e))$. The map $r_F$ induces a map from the set of paths on $\Ga(\ca)$ to the set of paths on $\Ga(\ca,F)$; moreover, $r_F$ takes positive paths to positive paths, and positive minimal paths to positive minimal paths. 

\begin{lemma}
	\label{lem:easy injective}
	\textup{(\cite{deligne}*{Proposition 1.30}).}
	The natural map $\G(F)\to \G$ is injective and preserves the lattice structure. 
	Moreover, the image of $\G^+(F)$ is contained in $\G^+$.
\end{lemma}

The lemma follows from the existence of $r_F$. Although the preservation of the lattice structure
is not mentioned explicitly in  \cite{deligne}, it follows easily from existence of $r_F$ and the fact that a positive $F$-path cannot be equivalent to a path which is not an $F$-path (by Lemma~\ref{lem:intersection number}).

For a vertex $x\in\Ga(\ca)$ and a face $F$ of $Z(\ca)$, denote by $\Sigma_{x\to}(F)$ (resp. $\Sigma_{\to x}(F)$)  the collection of edges in $\Ga(\ca)$ whose image under $\pi$ is contained in $F$ and whose source (resp. target) is $x$. 
Let $\ant(x,F)$ denote the antipodal vertex to $x$ in $F$. (Any zonotope is centrally symmetric.) For $y=\ant(x,F)$ define $\Delta^{x}(F)$ (resp. $\Delta_{x}(F)$) to be the morphism represented by a minimal positive path from $y$ to $x$ (resp. $x$ to $y$). The equivalence classes of such positive paths of this form are called \emph{Garside elements}. For an integer $k\ge 1$, define 
$$(\Delta_x(F))^k:=\underbrace{\Delta_x(F)\Delta_y(F)\Delta_x(F)\cdots}_{k \text{ times}}.$$
To simplify notation, put 
\begin{align*}
\Delta^x=\Delta^x(&Z(\ca)), \quad \Sigma_{x\to}=\Sigma_{x\to}(Z(\ca)),\\
\Delta_x=\Delta_x(&Z(\ca)), \quad\Sigma_{\to x}=\Sigma_{\to x}(Z(\ca)),
\end{align*}
and set $\ant(x)=\ant(x,Z(\ca))$.

\begin{lemma}	\textup{(\cite{deligne})}.
	\label{lem:longest}
	The following statements are true.
	\begin{enumerate1}
		\item For any vertex $x\in \Ga(\ca)$ and for any path $g$ from $x$ to $y$, we have $(\Delta_x)^2g\eq g(\Delta_y)^2$.
		\item Let $y=\ant(x)$. Then $\Delta_x\su e$ for any $e\in \Sigma_{\to y}$.
		\item Suppose $\Sigma_{x\to}=\{e_1,e_2,\ldots,e_k\}$ and $\Sigma_{\to x}= \{e'_1,e'_2,\ldots,e'_k\}$. Then $\Delta_x\eq e_1\vee_p e_2\vee_p\cdots\vee_p e_k$ and $\Delta^x\eq e'_1\vee_s e'_2\vee_s\cdots\vee_s e'_k$.
		\item Let $f$ be any path with $x=s(f)$. Then there is an integer $k\ge 0$ and positive path $g$ such that $f\eq(\Delta_x)^{-2k}g\eq g(\Delta_{t(f)})^{-2k}$.
		\item Suppose $v$ is a positive path with $x=s(v)$ and $y=t(v)$. Then there is a positive path $u$ such that $\Delta^x v\eq u\Delta^y$.
	\end{enumerate1}	
\end{lemma}

The next lemma is a consequence of Lemma~\ref{lem:easy injective} and Lemma~\ref{lem:longest} (3).
\begin{lemma}
	\label{lem:longest face}
	For vertex $x\in \Ga(\ca)$, let $\Sigma_{x\to}(F)=\{e_1,e_2,\ldots, e_k
	\}$ and $\Sigma_{\to x}(F)= \{e'_1,e'_2,\ldots, e'_k\}$). Then 
	\[
	\Delta_x(F)= e_1\vee_p e_2\vee_p\cdots\vee_p e_k \quad \text{and}\quad \Delta_x(F)=e'_1\vee_s e'_2\vee_s\cdots\vee_s e'_k.
	\]	
where $F$ is the face of $Z(\ca)$ spanned by $\{\bar e_1,\ldots,\bar e_k\}$, where the $\bar e_i$ ($=\pi(e_i)$) are as defined in the first paragraph of Subsection~\ref{subsec:deligne groupoid}
	\end{lemma}

\subsection{Elementary $B$-segments}
\begin{definition}
	\label{def:adj}
	Parallel faces $F$ and $F'$ of $Z(\ca)$ are \emph{adjacent} if $F\neq F'$ and if they are contained in a face $F_0$ with $\dim(F_0)=\dim(F)+1$.
\end{definition}
\begin{lemma}
	\label{lem:ant}
	Let $F,F'$ and $F_0$ be as in Definition~\ref{def:adj}.  
	\begin{enumerate1}
		\item For any vertex $x\in F$,  $\ant(x,F_0)\in F'$;
		\item If $p:F\to F'$ is parallel translation, then $\ant(p(x),F')=\ant(x,F_0)$ .
	\end{enumerate1}
\end{lemma}

\begin{definition}
	\label{def:elementary segment}
	Let $B\in\cq(\ca)$. Let $F$ and $F'$ be two adjacent parallel faces of $Z(\ca)$ that are dual to $B$. An \emph{elementary $B$-segment}, or an \emph{$(F,F')$-elementary $B$-segment} is a minimal positive path from a vertex $x\in F$ to $x'=p(x)\in F'$, where $p:F\to F'$ is parallel translation.
\end{definition}

\begin{lemma}
	\label{lem:elementary segment}
	Let $B,F,F'$ and $F_0$ be as in Definition~\ref{def:elementary segment}. 
	\begin{enumerate1}
		\item For any vertex $x\in F$ and $y=\ant(x,F_0)\in F'$, $\Delta_x(F_0)(\Delta^y(F'))^{-1}\eq\Delta^y(F_0)(\Delta^y(F'))^{-1}$ is an elementary $B$-segment.
		\item Any elementary $B$-segment can be written as $\Delta_x(F_0)(\Delta^y(F'))^{-1}$ for some vertex $x\in F$.
		\item Let $u$ be an elementary $B$-segment. Let $x'=t(u)$. Then $(\Delta_x(F))^ku\eq u(\Delta_{x'}(F'))^k$ for any positive even integer $k$. 
		\item Let $e\in \Sigma_{x\to}(F)$ (resp. $e\in \Sigma_{\to x}(F)$) and let $e'$ be the oriented edge in $F'$ that is  parallel to $e$. Then there is an elementary $B$-segment $v$ such that $ev\eq ue'$ (resp. $eu\eq ve'$). 
		\item A positive path representing an elementary $B$-segment does not cross any hyperplane dual to an edge of $F$.
	\end{enumerate1}
\end{lemma}

\begin{proof}
	Let $p:F\to F'$ be parallel translation. By Lemma~\ref{lem:gate}, there is a minimal positive path from $x$ to $y$ passing through $p(x)$ such that the segment from $p(x)$ to $y$ is an $F'$-path. So, statements (1) and (2) follow from Lemma~\ref{lem:ant} (2). For (3), let $z=\ant(x,F)$ and $z'=\ant(x',F')$. Then $z'=p(z)$. Let $v$ be an elementary $B$-segment from $z$ to $z'$. By Lemma~\ref{lem:gate}, both $u\Delta_{x'}(F')$ and $\Delta_x(F)v$ are minimal positive paths with same endpoints. Thus, $u\Delta_{x'}(F')\eq\Delta_x(F)v$. Similarly, $v\Delta_{z'}(F')\eq\Delta_z(F)u$. Statement (3) now follows by repeatedly applying these two equations.
	Similarly, statement (4) can be deduced from Lemma~\ref{lem:gate}. Statement (5) is immediate.
\end{proof}

\begin{lemma}
	\label{lem:lcm}
	Take two pairs of adjacent parallel faces $(F_1,F)$ and $(F_2,F)$ dual to the subspace $B$. For $i=1,2$, let $u_i$ be an $(F_i,F)$-elementary $B$-segment. Suppose $t(u_1)=t(u_2)=x$. Then $u_1\vee_s u_2\eq\Delta^y(F')(\Delta_x(F))^{-1}$ where $F'$ is the smallest face of $Z(\ca)$ containing $F,F_1$ and $F_2$ and $y=\ant(x,F')$. Moreover, $\dim(F')=\dim(F)+2$ unless $F_1=F_2$. 
\end{lemma}

In particular, $u_1\vee_s u_2$ is represented by a minimal positive path which is a concatenation of elementary $B$-segments.

\begin{proof}
	For $i=1,2$, let $F'_i$ be the face containing $F\cup F_i$ such that $\dim(F'_i)=\dim(F_i)+1$. Let $e_i$ be the last edge of $u_i$. Then $F'_i$ is spanned by $F$ and $e_i$ (by Lemma~\ref{lem:span}). Moreover, $F'$ is the face spanned by $F,e_1$ and $e_2$. Also, $F'=F'_1=F'_2$ if and only if $e_1=e_2$; otherwise, $\dim(F')=\dim(F)+2$. Let $y=\ant(x,F)$. By Lemma~\ref{lem:longest face},
	\begin{align*}
	\Delta^y(F'_1)\vee_s\Delta^y(F'_2)\eq (
	\vee_s\Sigma_{\to y}(F'_1))\vee_s (
	\vee_s\Sigma_{\to y}(F'_2))\eq \vee_s\Sigma_{\to y}(F')\eq \Delta^y(F')
	\end{align*}
	where $\vee_s\Sigma_{\to y}(F'_1)$ denotes the least common multiple of all edges in $\Sigma_{\to y}(F'_1)$ (where ``least'' is with respect to the suffix order).

	By Lemma~\ref{lem:elementary segment} (1), $u_i\Delta_x(F)\eq \Delta^y(F'_i)$ for $i=1,2$. Thus, $$(u_1\vee_s u_2)\Delta_x(F)\eq (u_1\Delta_x(F))\vee_s (u_2\Delta_x(F))\eq \Delta^y(F'_1)\vee_s\Delta^y(F'_2)\eq \Delta^y(F').$$
	The lemma follows.
\end{proof}

\subsection{Computation of centralizers}
In this subsection, we modify some computations of centralizers of parabolic subgroups of spherical Artin groups from \cite[Lemma 5.6 and Theorem 5.2]{Parisparabolic} to the context of  Deligne groupoids.
\begin{lemma}
	\label{lem:tail}
	Take a face $F'\subset Z(\ca)$. Let $f$ be a positive path in $\Ga(\ca)$. For a positive integer $k$, consider $g=f(\Delta_{x'}(F'))^k$ where $x'=t(f)$. Suppose $g\su e$ for an edge $e$ with $\bar e\nsubseteq F'$. Then there exists an elementary $B$-segment $u$ such that $f\su u$.
\end{lemma}

\begin{proof}
	Let $x'_m=t(f(\Delta_{x'}(F'))^m)$.	Let $F'_1$ be the face spanned by $F'$ and $\bar e$. We prove by induction on $i$ that if $1\le i\le k$, then there exists a positive path $f_i$ and an elementary $B$-segment $v_i$ such that $$f(\Delta_{x'}(F'))^{k-i}\eq f_iv_i$$
	and $\bar{v_i}\subset F'_1$.

	First consider the case $i=1$. Then $e\vee_s \Delta^{x'_k}(F')\eq \Delta^{x'_k}(F'_1)$. Thus, $f(\Delta_{x'}(F'))^k\eq f_1\Delta^{x'_k}(F'_1)$ for some positive path $f_1$. 
	Then $$f(\Delta_{x'}(F'))^{k-1}\eq f_1\Delta^{x'_k}(F'_1)(\Delta^{x'_k}(F'))^{-1}\eq f_1v_1.$$
	Next suppose that $f(\Delta_{x'}(F'))^{k-i}\eq f_iv_i$. 
	As $v_i$ ends with an edge $e_i$ such that $\bar e_i\subset F'_1$ and $\bar e_i\nsubseteq F'$, we have that $e_i\vee_s \Delta^{x'_{k-i}}(F')\eq \Delta^{x'_{k-i}}(F'_1)$. Thus, $f_i(\Delta_{x'}(F'))^{k-i}\eq f_{i+1}\Delta^{x'_{k-i}}(F'_1)$ for a positive path $f_{i-1}$. Then $$f_i(\Delta_{x'}(F'))^{k-i-1}\eq f_{i+1}\Delta^{x'_{k-i}}(F'_1)(\Delta^{x'_{k-i}}(F'))^{-1}\eq f_{i+1}v_{i+1}.$$
	This completes the proof.
\end{proof}

\begin{lemma}
	\label{lem:decomposition}
	Given two parallel faces $F$, $F'$ of $Z(\ca)$,  let $f$ be a positive path in $\Ga(\ca)$ from $x\in F$ to $x'\in F'$ satisfying the following. 
	\begin{enumerate1}
		\item $f(\Delta_{x'}(F'))^k\eq hf$, for an even integer $k\ge 2$ and a positive path $h$.
		\item There does not exist an edge $e\in \Sigma_{\to x'}(F')$ such that $f\succcurlyeq e$.
	\end{enumerate1}
	Then $f\eq u_1u_2\cdots u_n$ where each $u_i$ is an elementary $B$-segment and $B\in\cq(\ca)$ is dual to $F$. Moreover, $h=(\Delta_x(F))^k$. 
\end{lemma}

\begin{proof}
	We use induction  on the length of $f$. Let $e$ be the last edge of $f$. By assumption, $\bar e\nsubseteq F$. Let $F'_1$ be the face spanned by $F'$ and $e$. By Lemma~\ref{lem:tail}, $f\eq f_kv_k$ where $f_k$ is positive and $v_k$ is an elementary $B$-segment. Let $x''=t(f_k)$ and let $F''$ be the face such that $F''$ is parallel to $F$ and $x''\in F''$. By Lemma~\ref{lem:elementary segment} (3),
	\begin{align*}
	f_k(\Delta_{x''}(F''))^kv_k&\eq f_kv_k(\Delta_{x'}(F'))^k\eq f(\Delta_{x'}(F'))^k\\
	&\eq hf\eq hf_kv_k.
	\end{align*}
	Thus, $f_k(\Delta_{x''}(F''))^k\eq hf_k$. Moreover, there does not exist an edge $e\in \Sigma_{\to x''}(F)$ such that $f_k\succcurlyeq e$, for otherwise, by Lemma~\ref{lem:elementary segment} (4), we would have $f\succcurlyeq e'$ for $e'$ parallel to $e$. The lemma follows by induction. The last statement of the lemma follows from Lemma~\ref{lem:elementary segment} (3) and induction.
\end{proof}

\begin{corollary}
	\label{cor:concatenation}
	Let $F\subset F'$ be two faces of $Z(\ca)$. Suppose $B\subset \cq(\ca)$ is dual to $F$. Let $y\in F$ be a vertex. Then $\Delta^y(F')(\Delta^y(F))^{-1}$ is equivalent to a concatenation of elementary $B$-segments.
\end{corollary}

\begin{proof}
	Let $f=\Delta^y(F')(\Delta^y(F))^{-1}$ and $x=\ant(y,F)$. By Lemma~\ref{lem:longest}, $f(\Delta^y(F))^2 f^{-1}= \Delta^y(F')\Delta_y(F)f^{-1}=\Delta^y(F')\Delta_y(F)\Delta^y(F) (\Delta^y(F'))^{-1}$ is positive, so the first condition of Lemma~\ref{lem:decomposition} holds. As $\Delta^y(F')\eq f\Delta^y(F)$ is positive and minimal, $f$ cannot cross any hyperplane dual to $F$. So the second condition of Lemma~\ref{lem:decomposition} also holds. The corollary follows.
\end{proof}

\begin{corollary}
	\label{cor:centralizer}
	Let $F$ be a face of $Z(\ca)$ and let $x\in F$ be a vertex. Let $f\in \g_x$ be such that $f(\Delta_x(F))^2=(\Delta_x(F))^2f$. Let $\g_{x,F}$ be the collection of $F$-paths with both source and target equal to $x$. Then there exists an  integer $k\ge 0$, and a collection of elementary $B$-segments (where $B\in \cq(\ca)$ is dual to $F$) $\{u_i\}_{i=1}^k$ and a positive path in $\g_{x,F}$ such that $f=(\Delta_x)^{-2k}u_1u_2\cdots u_kv$.
\end{corollary}

\begin{proof}
	By Lemma~\ref{lem:longest}, we have $f=(\Delta_x)^{-2k}g$ where $g$ is positive. Then $g(\Delta_x(F))^2=(\Delta_x(F))^2g$. We can write $g=hv$ where $v$ is a positive $F$-path, $h$ is positive and there does not exist edge $e$ with $\bar e\subset F$ and $h\su e$. Let $x'=s(v)=t(h)$. Then $v(\Delta_x(F))^2=(\Delta_{x'}(F))^2v$. It follows that $h(\Delta_{x'}(F))^2=(\Delta_x(F))^2h$. By Lemma~\ref{lem:decomposition}, $h$ is a concatenation of elementary $B$-segments. It follows from the definition of a  $B$-segment that $x'=x$. Thus, $v\in \g_{x,F}$.
\end{proof}

\subsection{Injectivity of the restriction arrangement groupoid}
In this subsection we study how the Deligne groupoid of a restriction arrangement sits inside the ambient Deligne groupoid. This will show how to push $\pi_1(\cm(\ca^B\otimes \cc)$ into $\pi_1(\cm(\ca\otimes \cc)$.
\begin{lemma}
	\label{lem:adj}
	Let $B\in\cq(\ca)$. 	
	\begin{enumerate1}
		\item There is a natural one-to-one  correspondence between chambers of $\ca^B$ and the collection of faces of $Z(\ca)$ that are dual to $B$. 
		\item Two chambers of $\ca^B$ are adjacent if and only if the associated faces of $Z(\ca)$ (as in (1) above) are adjacent (see Definition~\ref{def:adj}).
	\end{enumerate1}
\end{lemma}

\begin{definition} 
	\label{def:Phi}
	We define a map $\phi:\Ga(\ca^B)\to\Ga(\ca)$ as follows. First choose a vertex $x\in\Ga(\ca^B)$ and let $F$ be the face  of $Z(\ca)$ associated with $x$ as in the previous lemma. Then choose a vertex $y$ in $F$ and set $\phi(x)=y$. For any other vertex $x'\in\Ga(\ca^B)$, let $F'$ be the face of $Z(\ca)$ associated with $x'$ and let $p:F\to F'$ be  parallel translation. Define $\phi(x')$ to be $p(y)$. Let $e$ be the edge of $\Ga(\ca^B)$ from $x_1$ to $x_2$. Define $\phi(e)$ be a positive minimal path from $\phi(x_1)$ to $\phi(x_2)$. (Such a positive minimal path always exists; however, it might not be unique. If it is not, choose one arbitrarily.)
\end{definition}

It is clear that $\phi:\Ga(\ca^B)\to\Ga(\ca)$ induces maps $\g^+(\ca^B)\to\g^+(\ca)$ and $\g(\ca^B)\to\g(\ca)$, both of which we also denote by $\phi$.

Let $y'$ be a vertex of $\Ga(\ca)$ such that $y'=\phi(y)$ for some $y\in\Ga(\ca^B)$ (such a $y$ is unique). Let $F_{y'}$ be the face of $Z(\ca)$ associated with $y$ as in Lemma~\ref{lem:adj} (1).

\begin{lemma}
	\label{lem:injective}
	The map $\phi:\g(\ca^B)\to\g(\ca)$ is injective.
\end{lemma}

\begin{proof}
	We first show that $\phi:\g^+(\ca^B)\to\g^+(\ca)$ is injective. Let $\hat f_1,\hat f_2$ be two positive paths in $\Ga(\ca^B)$ and let $f_1,f_2$ be their image under $\phi$. Then $f_1=u_1u_2\cdots u_n$, where each $u_i$ is an elementary segment coming from an edge of $f_1$. Similarly, $f_2$ can be written as $f_2=v_1v_2\cdots v_m$.
	
	Suppose $f_1\eq f_2$. We need to show that $\hat f_1\eq_B \hat f_2$, where $\eq_B$ refers to equivalent paths in $\Ga(\ca^B)$. This will be proved by induction on $c=\ell(f_1)+\ell(f_2)$, where $\ell(f)$ denotes the length of $f$.
	The case $c=0$ is immediate, as $\ell(\phi(\hat f))>0$ if and only if $\ell(\hat f)>0$. Next we consider the general case $c\ge 0$. 	
	Let $w=u_n\vee_s v_m$. Then there is positive path $p$ such that $f_1\eq pw\eq f_2$. Let $w_1$ and $w_2$ be positive paths such that $w\eq w_1u_n\eq w_2v_m$. By Corollary~\ref{cor:concatenation} and Lemma~\ref{lem:lcm}, $w_1$ and $w_2$ are equivalent to concatenations of elementary $B$-segments. By Lemma~\ref{lem:elementary segment} (3), both $f_1$ and $w$ satisfy condition (1) of Lemma~\ref{lem:decomposition}; so, the same holds for the positive path $p$. By Lemma~\ref{lem:elementary segment} (5), both $f_1$ and $w$ do not cross any hyperplane dual to $F$, and so, the same holds for $p$ by Lemma~\ref{lem:intersection number}.  By Lemma~\ref{lem:decomposition}, $p$ is equivalent to a concatenation of elementary $B$-segments. Thus, we can find positive paths $\hat w_1$ and $\hat p$ of $\Ga(\ca^B)$ such that $\phi(\hat w_1)\eq w_1$ and $\phi(\hat p)\eq p$. We delete the last edge $\hat e_1$ of $\hat f_1$ to obtain $\hat g_1$. Then $\phi(\hat g_1)\eq\phi(\hat p\hat w_1)\eq u_1u_2\cdots u_{n-1}$. We can assume by induction that $\hat g_1\eq_B \hat p\hat w_1$. Define $\hat w_2,\hat e_2$ and $\hat g_2$ similarly and then conclude that $\hat g_2\eq_B\hat p\hat w_2$. By Lemma~\ref{lem:lcm}  both $\hat w_1\hat e_1$ and $\hat w_2\hat e_2$ are positive minimal paths in $\Gamma(\ca^B)$. Thus, $\hat w_1\hat e_1\eq_B \hat w_2\hat e_2$. Then $$\hat f_1\eq_B \hat g_1\hat e_1\eq_B \hat p\hat w_1\hat e_1\eq_B \hat p\hat w_2\hat e_2\eq_B \hat g_2\hat e_2\eq_B \hat f_2.$$
	
	Next take paths $\hat f_1,\hat f_2$ on $\Gamma(\ca^B)$ not necessarily positive and let $f_1,f_2$ denote their $\phi$-images. Suppose $f_1\eq f_2$. By Lemma~\ref{lem:longest} and Remark~\ref{rmk:inherit}, there is a $k\ge 0$ such that for $i=1,2$, $\hat f_i\eq_B (\Delta_{s(\hat f_i)})^{-2k}\cdot\hat g_i$ with $\hat g_i$ positive.  Since $f_1\eq f_2$ we deduce that $\phi(\hat g_1)\eq \phi(\hat g_2)$. By the previous discussion, $\hat g_1\eq_B\hat g_2$. Hence, $\hat f_1\eq_B \hat f_2$, as required.
\end{proof}

Let $F$  be a face of $Z(\ca)$ and let $\Gamma(\ca,F)$ be the graph defined in Section~\ref{subsec:face}. For the product arrangement $\ca_{A/B}\times \ca^B$ (which is also simplicial by Remark~\ref{rmk:inherit}), we can identify $\Gamma(\ca_{A/B}\times \ca^B)$ with the 1-skeleton of $\Gamma(\ca,F)\times \Gamma(\ca^B)$. Thus, each vertex of $\Gamma(\ca_{A/B}\times \ca^B)$ corresponds to a vertex of $\Gamma(\ca)$. This induces a map
$$
\theta:\Gamma(\ca_{A/B}\times \ca^B)\to \Gamma(\ca)
$$
so that for each vertex $x\in F$, the restriction of $\theta$ to $\{x\}\times \Gamma(\ca^B)$  is equal to the map $\phi$ in Definition~\ref{def:Phi}.
\begin{corollary}
	\label{cor:injective and centralizer}
	The map $\theta$ satisfies the following.
	\begin{enumerate1}
		\item The homomorphism of groupoids $\theta_*:\g(\ca_{A/B}\times \ca^B)\to\g(\ca)$, that is induced by $\theta$,  is injective.
		\item Let $x'\in \Gamma(\ca_{A/B}\times \ca^B)$ be a vertex. Let $x=\theta(x')$ and let $F_x$ be the face containing $x$ and parallel to $F$. Then $\theta_*(\g_{x'}(\ca_{A/B}\times \ca^B))$ is the centralizer of $(\Delta_x(F_x))^2$ in $\g_x(\ca)$.
		\item Let $U=\langle(\Delta_x(F_x))^2\rangle$ and $G=\g_x(\ca)$. Then $C_G(U)=N_G(U)=\theta_*(\g_{x'}(\ca_{A/B}\times \ca^B))$.
	\end{enumerate1}
\end{corollary}
The centralizer $U$ in $G$ is denoted by $C_G(U)$ and the normalizer denoted by $N_G(U)$.
\begin{proof}
	Note that $\theta$ takes minimal positive paths to minimal positive paths; so, $\theta_*$ is well-defined. Moreover, $\g(\ca_{A/B}\times \ca^B)\cong \g(\ca_{A/B})\times \g(\ca^B)$. So, the first assertion follows from Lemmas~\ref{lem:easy injective} and \ref{lem:injective}. By Lemma~\ref{lem:longest} (1), $\theta_*(\g_{x'}(\ca_{A/B}\times \ca^B))$ is contained in the centralizer of $(\Delta_x(F_x))^2$. The containment in the other direction follows from Corollary~\ref{cor:centralizer}. For (3), we need to show that $(\Delta_x(F_x))^2$ is not conjugate to its inverse; however, this follows from Lemma~\ref{lem:intersection number} since two conjugate paths have the same intersection number with each hyperplane in $\ca$.
\end{proof}

\subsection{Simultaneously conjugating mutually commuting Dehn twists}

\begin{definition}\label{def:dehn twists}
	Take a base vertex $x_0\in\Ga(\ca)$. Let $F$ be an irreducible face of $Z(\ca)$ and let $x_F=\prj_F(x)$ (defined in Lemma~\ref{lem:gate}). Let $h_F$ be a minimal positive path from $x_0$ to $x_F$. The \emph{standard Dehn twist} associated with $F$, denoted $\cz_F$, is defined to be $h_F(\Delta_{x_F}(F))^2h^{-1}_F$. 
\end{definition}

\begin{lemma}
	\label{lem:standalizer}
	Let $h$ be a positive path from $x_0$ to $x_F$. We write $f=h\cz_F h^{-1}$ in its $pn$-normal form $ab^{-1}$. Then there is  a $F'$ parallel to $F$ such that $b^{-1}fb\eq (\Delta_{x}(F'))^2$ for a vertex $x\in F'$.
\end{lemma}

This observation is taken from \cite[Theorem 4]{cumplido} in the case of spherical Artin groups. The same proof given there also works for the Deligne groupoid and we give it below.
\begin{proof}
	Let $g=hh_F$. Then $f=g(\Delta_{x_F}(F))^2 g^{-1}$. If there does not exist an edge $e$ such that $g(\Delta_{x_F}(F))^2\su e$ and $g\su e$, then we are done by Lemma~\ref{lem:pn form}. Suppose such $e$ exists. If $\bar e\subset F$, then $e(\Delta_{x_F}(F))^2 e^{-1}\eq(\Delta_{t(e)}(F))^2$ by Lemma~\ref{lem:longest} and we can shorten $g$. If $\bar e\nsubseteq F$, then Lemma~\ref{lem:tail} implies that $g\eq g_1u$ where $g_1$ is positive and $u$ is an elementary $B$-segment where $B$ is the subspace dual to $F$. By Lemma~\ref{lem:elementary segment} (3), $u(\Delta_{t(e)}(F))^2 u^{-1} \eq(\Delta_{s(u)}(F_1))^2$ for a face $F_1\parallel F$. So again, we can shorten $g$ again. The lemma follows by repeatedly applying this procedure.
\end{proof}

Alternatively, one could  define $\cz_F$ by taking an arbitrary vertex $x\in F$ (not necessarily $x_F$), taking a minimal positive path $h$ from $x_0$ to $x$, and considering $h(\Delta_x(F))^2h^{-1}$. We claim that $\cz_F\eq h(\Delta_x(F))^2h^{-1}$. Indeed, by \cite[Section 2]{deligne}, there exists a minimal path $h'=h_1h_2$ such that $h\eq h'$, $t(h_1)=s(h_2)=x_F$ and $\bar h_2\subset F$. We can then use the same argument  as in the proof of Lemma~\ref{lem:standalizer} to deduce the claim.

The following lemma is proved in \cite{cgw} for spherical Artin groups, where the argument is based on solution of the conjugacy problem for Garside groups. We give a shorter proof based on arguments from \cite{Parisparabolic}.
\begin{lemma}
	\label{lem:conjugate}
	Let $\{F_i\}_{i=1}^k$ be a collection of faces of $Z(\ca)$. Consider a mutually commuting collection $\{g_i:=h_i\cz_{F_i} h^{-1}_i\}_{i=1}^k$ where $h_i\in \g_{x_0}$. Then there exist an element $g\in \g_{x_0}$ such that for each $i$, $gg_ig^{-1}=\cz_{F'_i}$ for some face $F'_i\subset Z(\ca)$ parallel to $F_i$. Moreover, there is a vertex $x\in \Ga(\ca)$ such that $x\in\bigcap_{i=1}^kF'_i$.
\end{lemma}

\begin{proof}
	The proof is by induction on $k$. The case $k=1$ is clear. Assume by induction that $h_i$ is the trivial element for $1\le i\le k-1$,  and assume moreover, that there is a vertex $x\in\Ga(\ca)$ and a minimal positive path $h$ from $x_0$ to $x$ such that $\cz_{F_i}=h(\Delta_x(F_i))^2h^{-1}$ for each $1\le i\le k-1$. Define $S_0=\{\cz_{F_1},\cdots,\cz_{F_{k-1}},h_k\cz_{F_k}h^{-1}_k\}$. Then any two elements from the set $S_1=\{(\Delta_x(F_1))^2,\cdots,(\Delta_x(F_{k-1}))^2,h^{-1}h_k\cz_{F_k}h^{-1}_kh\}$ commute.
	
	By Lemma~\ref{lem:longest}, we can assume $f:=h^{-1}h_k$ is positive. Let $ab^{-1}$ be the $pn$-normal form of $f\cz_{F_k}f^{-1}$. For each $i$ with $1\le i\le k-1$, we have $(\Delta_x(F_i))^2ab^{-1}=ab^{-1}(\Delta_x(F_i))^2$. Thus, $(\Delta_x(F_i))^2a\cdot [(\Delta_x(F_i))^2b]^{-1}\eq a\cdot b^{-1}$. It follows from Lemma~\ref{lem:pn form} that $a^{-1}(\Delta_x(F_i))^2a\eq b^{-1}(\Delta_x(F_i))^2b$ is positive. Decompose $b$ as $b=b_1b_2$, where $b_2$ is a maximal positive $F_i$-path. By applying Lemma~\ref{lem:longest} (1) to $b_2$ and Lemma~\ref{lem:decomposition} to $b_1$, we deduce that $b^{-1}(\Delta_x(F_i))^2b\eq(\Delta_y(F'_i))^2$ where $F'_i\parallel F_i$ and $y=t(b)$. By Lemma~\ref{lem:standalizer}, $b^{-1}(f\cz_{F_k}f^{-1})b\eq (\Delta_y(F'_k))^2$ with $F'_k$ parallel to $F_k$. Thus, $b^{-1}h^{-1}S_0hb=b^{-1}S_1b=\{(\Delta_y(F'_i))^2\}_{i=1}^k$. Take a minimal positive path $h'$ from $x_0$ to $y$ and let $g=h'b^{-1}h^{-1}$. Then $g S_0 g^{-1}=\{\cz_{F'_i}\}_{i=1}^k$.
\end{proof}

\begin{lemma}
	\label{lem:commute}
	Let $F_1$ and $F_2$ be two irreducible faces of $Z(\ca)$. Then $\cz_{F_1}$ and $\cz_{F_2}$ commute if and only if $F_1\perp F_2$ or if $F_1$ and $F_2$ are comparable. 
	\end{lemma}

\begin{proof}
	By Lemma~\ref{lem:conjugate}, $F_1\cap F_2\neq \emptyset$. We assume without loss of generality that the base point $x_0$ is contained in $F_1\cap F_2$, and that $Z(\ca)$ is spanned by $F_1$ and $F_2$. Let $f=(\Delta_{x_0}(F_1))^2(\Delta_{x_0}(F_2))^2\eq (\Delta_{x_0}(F_2))^2(\Delta_{x_0}(F_1))^2$. By Lemma~\ref{lem:longest}, $e\pr f$ for any $e\in \Sigma_{x_0\to}$; hence, $\Delta_{x_0}\pr f$. Let $\ch$ (resp. $\ch_i$) be the collection of hyperplanes dual to edges of $Z(\ca)$ (resp. $F_i$). Let $\ch_f$ be the collection of hyperplanes crossed by $f$. Then $\ch_f=\ch_1\cup\ch_2$. By Lemma~\ref{lem:intersection number}, $\ch_{\Delta_{x_0}}\subset\ch_f$. As $\ch_{\Delta_{x_0}}=\ch$, we have that $\ch=\ch_1\cup\ch_2$. Thus, if $F_1\cap F_2$ is a vertex, then $\ch=\ch_1\sqcup\ch_2$, which implies $F_1\perp F_2$.

	Next  assume $F:=F_1\cap F_2$ contains an edge, then $Z(\ca)$ is irreducible by Corollary~\ref{cor:irreducible}. Suppose $B$ is the subspace dual to $F$. Let $\theta:\Gamma(\ca_{A/B}\times \ca^B)\to\Gamma(\ca)$ be the map defined before Corollary~\ref{cor:injective and centralizer}. Let $\bar x_0$ be such that $\theta(\bar x_0)=x_0$. Let $\Delta_{\bar x_0}(\ca_{A/B})$ be the Garside element of the $\ca_{A/B}$-factor.
	
	An edge $e$ of $\Sigma_{x_0\to}(F_1)- \Sigma_{x_0\to}(F)$ gives rise to an elementary $B$-segment contained in the face spanned by $e$ and $F$ that corresponds to an edge $\bar e$ in $\Gamma(\ca^B)$. Let $F'_1$ be the face of $Z(\ca^B)$ spanned by all such edges coming from elements of $\Sigma_{x_0\to}(F_1) - \Sigma_{x_0\to}(F)$. Define $F'_2$ similarly. Then $F'_1\cap F'_2$ is a vertex, and $Z(\ca^B)$ is spanned by $F'_1$ and $F'_2$. 
	
	It is readily verified that $\theta((\Delta_{\bar x_0}(\ca_{A/B}))^2(\Delta_{\bar x_0}(F'_i))^2)\eq (\Delta_{x_0}(F_i))^2$ for $i=1,2$. By Corollary~\ref{cor:injective and centralizer} (1), $(\Delta_{\bar x_0}(F'_1))^2(\Delta_{\bar x_0}(F'_2))^2\eq (\Delta_{\bar x_0}(F'_2))^2(\Delta_{\bar x_0}(F'_1))^2$ in $\Gamma(\ca^B)$. By the argument in the previous paragraph, $Z(\ca^B)\cong F'_1\times F'_2$. By Lemma~\ref{lem:restriction irreducible}, either $F'_1$ or $F'_2$ has to be a point. Thus, $F_1\subset F_2$, or $F_2\subset F_1$.
\end{proof}

\section{The Deligne groupoid, the Deligne complex, and simple connectivity of the curve complex}
\label{sec:sc}

Throughout the this section we assume $\ca$ is irreducible.
\subsection{From the Deligne complex to the curve complex} 
\label{ss:curve complex groupoid}
Here we define a version of the curve complex without using bordifications: instead, it is defined in terms of the Deligne groupoid.  Choose a base vertex $x_0\in \Ga(\ca)$ and let $G=\g_{x_0}$ denote the isotropy group at $x_0$. A \emph{standard $\mathbb Z$-subgroup} of $G$ is a conjugate of a standard Dehn twist associated to a proper irreducible face of $Z(\ca)$. The \emph{groupoidal curve complex} $\cgro$ ($=\cgro(\ca)$) is defined as follows. The vertices of $\cgro$ are in one-to-one correspondence with standard $\mathbb Z$-subgroups of $G$ and a set of vertices spans a simplex if the associated $\mathbb Z$-subgroups mutually commute. From this we see that $\cgro$ is a flag complex. The group $G$ acts by conjugation on the set of standard $\mathbb Z$-subgroups and this induces  $G\act\cgro\,$.

\begin{remark}
	In this section we only need to consider $\cgro(\ca)$ when $\ca$ is irreducible.
	However, if $\ca$ has more than one  irreducible factor, say $\ca=\ca_1\oplus \cdots\oplus \ca_l$, then, by Definition~\ref{d:I0}, $\ci_0(\ca)=\ci_0(\ca_1)*\cdots *\ci_0(\ca_l)$.  By analogy, in the general case it would make sense to define $\cgro(\ca)$ as $\cgro(\ca_1)*\cdots *\cgro(\ca_l)$, although we never use this definition.
\end{remark}

\begin{definition}
	The \emph{spherical Deligne complex} associated with $\ca$, denoted $\de'(\ca)$, is defined as follows. 
	Let $\sfan(\ca))$ denote the cellulation of the unit sphere in $A$ cut out by the hyperplanes of $\ca$. For each vertex $y\in \Ga(\ca)$, let $C_y$ be chamber containing $y$ (so that $C_y$ is a top-dimensional closed cell of $\sfan(\ca)$).  For each vertex $x\in\widetilde{\s(\ca)}$, put $C_x=C_{k(x)}$, where $k:\widetilde{\s(\ca)}\to\s(\ca)$ is the covering projection. The \emph{spherical Deligne complex} associated with $\ca$, denoted $\de'(\ca)$, is defined by: 
	\[
	\de'(\ca)=\bigsqcup_{x\in\vertex (\widetilde{\s(\ca)})} (x,C_x)/\sim.
	\]
	Here $\sim$ is the equivalence relation which glues adjacent chambers along codimension-one faces.  To be more explicit, suppose $F_i$ is a face of $C_{x_i}$ for $i=1,2$.  We identify $(x_1,F_1)$ with $(x_2,F_2)$  if $x_1$ and $x_2$ are adjacent in $\widetilde{\s(\ca)}$ and $F_1=F_2$;  the  equivalence relation $\sim$ is  generated by such identifications. There is an action $G\act \de'(\ca)$.
	
	The union of all $(x,C_x)$, with $x$ ranging over the vertices of a top-dimensional face of $\widetilde{\s(\ca)}$, gives rise to an embedded sphere in $\de'(\ca)$ called an \emph{apartment} of $\de'(\ca)$. The \emph{Deligne complex}, denoted $\de(\ca)$, is obtained from $\de'(\ca)$ by filling in independently each apartment of $\de'(\ca)$ by a disk of dimension equal to $\dim(A)$.
\end{definition}

\begin{lemma}
	Suppose $\ca$ is a central arrangement of hyperplanes in a real vector space $A$. Then $\de(\ca)$ is simply connected.
\end{lemma}
The above lemma follows from \cite{deligne} (see also \cite[p. 49]{Parisdeligne}) since the Deligne complex (up to subdivision) is isomorphic to the nerve of an open cover of the universal cover of $\cm(\ca\otimes\cc)$ where each open set of the cover is connected.

Each face of $\de'(\ca)$ can be written as $(x,F)$ where $x\in\vertex(\widetilde{\s(\ca)})$ and $F$ being a simplex of $\sfan(\ca)$. For a simplex $F$ of $\sfan(\ca)$, we define $K_F$ to be the standard subcomplex of $\s(\ca)$ associated with the face of $Z(\ca)$ dual to $F$ (cf. Subsection~\ref{ss:sal}). 
A standard subcomplex of $\s(\ca)$ is \emph{proper} if it is neither a vertex nor the entire zonotope $Z(\ca)$.  Similarly, a \emph{proper standard subcomplex} of $\widetilde{\s(\ca)}$ is a connected component of the inverse image of a proper standard subcomplex in $\s(\ca)$ under the covering map $k:\widetilde{\s(\ca)}\to\s(\ca)$. A standard subcomplex is \emph{nontrivial} if it is not a point. 
Then $(x,F)=(y,F)$ if and only if $x$ and $y$ are in the same standard subcomplex of $\widetilde{\s(\ca)}$ that projects to $K_F$. Thus, the vertices of the barycentric subdivision $b\de'(\ca)$ of $\de'(\ca)$ are in one-to-one correspondence with standard subcomplexes of $\widetilde {\s(\ca)}$. This gives rise to an alternative description of the simplicial complex $b\de'(\ca)$: the vertices of $b\de'(\ca)$ are in one-to-one correspondence with standard subcomplexes of $\widetilde{\s(\ca)}$, and a simplex in $b\de'(\ca)$ corresponds to a chain of standard subcomplexes in $\widetilde{\s(\ca)}$.  

Now suppose $\ca$ is irreducible. 
Let $\D(\ca)$ be the full subcomplex of $b\de'(\ca)$ spanned by vertices that correspond to proper standard subcomplexes. 
In other words, $\D(\ca)$ can be identified with the barycentric subdivision of $(n-2)$-skeleton of $\de'(\ca)$, where $n=\dim A$; that is to say, the interiors of top-dimensional simplices of $\de'(\ca)$ have been deleted. Thus, $\D(\ca)$ is simply connected whenever $n\ge 4$ (i.e., when $n-2\ge 2$).

We define a $G$-equivariant map $\rho:\D(\ca)\to \cgro$ as follows. Let $v\in \D(\ca)$ be a vertex.  Then $v$ corresponds a standard subcomplex $K\subset\widetilde{\s(\ca)}$. Let $K=\prod_{i=1}^m K_i$ be the decomposition of $K$ into irreducible standard subcomplexes. Then stabilizer of each $K_i$ is an irreducible parabolic subgroup, whose centralizer gives rise to a vertex $w_i\in\cgro$. The vertices $\{w_i\}_{i=1}^m$ span a simplex in $\cgro$ and $\rho$ sends $v$ to the barycenter of this simplex. Suppose $\{v_i\}_{i=1}^k$ is a collection of vertices of $\D(\ca)$ spanning a simplex $\gs$. Let $K_1\subset K_2\subset\cdots\subset K_k$ be the associated chain of standard subcomplexes of $\widetilde{\s(\ca)}$. Then $\{\rho(v_i)\}_{i=1}^k$ is contained in a simplex of $\cgro$ (the simplex associated with the product decomposition of $K_k$). Thus, we can extend $\rho$ linearly to a map into $\cgro$. N.B.\ $\rho$ is usually not a simplicial map.

\subsection{$\cgro$ is simply connected}\label{ss:cgro}
In this subsection we study the map $\rho$ in more detail in order to prove that $\cgro$ is simply connected when $n\ge 4$. As before, $\ca$ is assumed to be irreducible.
\begin{lemma}
	\label{lem:chain}
	Let $v$ and $v'$ be two vertices in $\D(\ca)$ corresponding to irreducible proper standard subcomplexes $K$ and $K'$ of $\widetilde{\s(\ca)}$ such that $\rho(v)=\rho(v')$. Then there exists a finite sequence of vertices $\{v_i\}_{i=1}^k\subset\D(\ca)$ with associated standard subcomplexes $\{K_i\}_{i=1}^k$ such that
	\begin{enumerate1}
		\item $v_1=v$ and $v_k=v'$;
		\item for any $i$, $\rho(v_i)=\rho(v)$;
		\item $K_i$ and $K_{i+1}$ are contained in a common standard subcomplex $C$ with $\dim(C)=\dim(K_i)+1$. 
	\end{enumerate1}
\end{lemma}

\begin{proof}
	Let $F$ (resp. $F'$) be the irreducible face of $Z(\ca)$ corresponding to $K$ (resp. $K'$). Let $\tilde g$ be an edge path in $\widetilde{\s(\ca)}$ connecting a vertex $\tilde x \in K $ and vertex $\tilde x'\in K'$. Let $g$ be the projection of $\tilde g$ into $\Ga(\ca)$ to get an edge path from from $x$ to $x'$. As $\rho(v )=\rho(v')$, we have $(\Delta_{x }(F))^2g\eq g(\Delta_{x'}(F'))^2$. By the same argument as inCorollary~\ref{cor:centralizer}, up to changing $x $ and $x'$ (which corresponds to replace $\tilde x$ and $\tilde x'$ by other vertices in the same standard subcomplex), we can assume $g$ is a positive path and that there does not exist edge $e$ of $\Ga(\ca)$ with $\bar e\in F'$ and $g\su e$. By Lemma~\ref{lem:decomposition}, $F $ and $F'$ are parallel and $g=g_1g_2\cdots g_k$, where each $g_i$ is an elementary $B$ segment ($B$ is the subspace dual to $F'$). For $1\le i\le k-1$, define $F_i$ to be the face of $Z(\ca)$ that is parallel to $F$ and contains $s(g_i)$. Let $\tilde g=\tilde g_1\tilde g_2\cdots\tilde g_{k-1}$ be the induced decomposition. For $1\le i\le k-1$, define $K_i$ to be the standard subcomplex of $\widetilde{\s(\ca)}$ associated to $F$ and containing the initial vertex of $\tilde g_i$. It follows from Lemma~\ref{lem:elementary segment} (3) that $\rho(v_i)=\rho(v_{i+1})$ for $1\le i\le k-1$. This implies assertion (2) of the lemma. For assertion (3), note that $F_i$ and $F_{i+1}$ are contained in a common face $\hat F$ of one  higher dimension. Let $C$ be the standard subcomplex of $\widetilde{\s(\ca)}$ associated with $\hat F$ such that $K_i\subset C$. Then clearly $K_{i+1}\subset C$.
\end{proof}

\begin{corollary}
	\label{cor:trivial pi1}
	Let $v$ and $v'$ be as in Lemma~\ref{lem:chain} with $\rho(v)=\rho(v')=w$. Suppose $\dim(A)\ge 4$. Then there is an edge path $\omega$ in $\D(\ca)$ connecting $v$ and $v'$ such that $\rho(\omega)$ is a loop representing the trivial element in $\pi_1(\cgro,w)$.
\end{corollary}

\begin{proof}
	Suppose $\dim(A)=n$. 
	Let $\{K_i\}_{i=1}^k$ be as in Lemma~\ref{lem:chain}. Put $n'=\dim(K_i)$. If $n'<n-1$, then the standard subcomplex $C$ in Lemma~\ref{lem:chain} (3) is proper. Hence, by Lemma~\ref{lem:chain} (3), there is an induced  edge path $\omega$ in $\D(\ca)$ connecting $v$ and $v'$. By Lemma~\ref{lem:chain} (2), the image $\rho(\omega)$ is trivial in $\pi_1(\cgro,w)$ due to its multiple back-tracking. Suppose $n'=n-1$. We claim there are irreducible proper standard subcomplexes $\hat K$ and $\hat K'$ of $\widetilde{\s(\ca)}$ such that
	\begin{enumerate1}
		\item $\hat K\subsetneq K$ and $\hat K'\subsetneq K'$;
		\item $\rho(\hat v)=\rho(\hat v')$ where $\hat v$ and $\hat v'$ are vertices of $\D(\ca)$ associated with $\hat K$ and $\hat K'$.
	\end{enumerate1}
	To establish this claim, let $F,F',g$ and $\tilde g$ be as in Lemma~\ref{lem:chain}.  Since $F$ is irreducible it follows from Lemma~\ref{lem:irreducible characterization} that  there is a 2-dimensional irreducible face $\hat F\subset F$. Note that $\hat F\neq F$ since $\dim(A)\ge 4$. Let $\hat F'$ be the face of $F'$ parallel to $\hat F$. Let $\hat K$ (resp. $\hat K'$) be the standard subcomplex of $K$ (resp. $K'$) associated with $\hat F$ (resp. $\hat F'$) such that $s(\tilde g)\subset \hat K$ (resp. $t(\tilde g)\subset\hat K'$). As $g$ is a concatenation of elementary $B$ segments, it follows from Lemma~\ref{lem:elementary segment} that $(\Delta_{x }(\hat F))^2g\eq g(\Delta_{x'}(\hat F'))^2$. Thus, $\rho(\hat v)=\rho(\hat v')$.
	
	By the previous case when $n'<n-1$, there is an edge path $\omega'$ connecting $\hat v$ and $\hat v'$ such that $\rho(\omega')$ is trivial in $\pi_1(\cgro,\rho(\hat v))$. Defining $\omega$ to be the concatenation of the edge $\overline{v\hat v}$, $\omega'$ and the edge $\overline{\hat v' v'}$, we see that it satisfies the requirements of the corollary.
\end{proof}

\begin{proposition}
	\label{prop:sc}
	Suppose $\ca$ an irreducible, central, simplicial real arrangement in $A$. If $\dim A\ge 4$, then $\cgro$ is simply connected.
\end{proposition}

\begin{proof}
	Let $\widetilde{\cgro}$ be the universal cover of $\cgro$ and $\pi:\widetilde{\cgro}\to\cgro$  the  covering projection.  The proposition will be proved by showing that $\pi$ is a homeomorphism.
	Let $\rho:\D(\ca)\to\cgro$ be the $G$-equivariant map defined previously. Since $\dim A\ge 4$, $\D(\ca)$ is simply connected. Hence, there is a lift $\tilde\rho:\D(\ca)\to\widetilde{\cgro}$ of $\rho$. 	Next we construct a section of $s:\cgro\to\widetilde{\cgro}$ of $\pi$. Choose a vertex $v\in\cgro$. Since $\tilde\rho$ is a lift of $\rho$, it follows from Corollary~\ref{cor:trivial pi1} that if two vertices $v_1$, $v_2$ of $\D(\ca)$ satisfy $\rho(v_1)=\rho(v_2)=w$ where $w$ is a vertex of $\cgro$, then $\tilde\rho(v_1)=\tilde\rho(v_2)$. Define $s(w)=\tilde\rho(v_1)$, and note that this  does not depend on the choice of $v_1$ in the preimage of $w$. Taking two adjacent vertices $w_1$ and $w_2$, we claim $s(w_1)$ and $s(w_2)$ are adjacent. By Lemma~\ref{lem:conjugate}, one of the following holds:
	\begin{enumerate1}
		\item There are proper standard subcomplexes $K_1$ and $K_2$ in $\widetilde{\s(\ca)}$ with one contained in the other so that for $i=1,2$, $\rho(v_{K_i})=w_i$,  where $v_{K_i}$ is the vertex of $\D(\ca)$ associated to $K_i$.
		\item There are proper standard subcomplexes $K_1,K_2$ and $K$ in $\widetilde{\s(\ca)}$ with $K=K_1\times K_2$ and with $
		\rho(v_{K_i})=w_i$ for $i=1,2$.
	\end{enumerate1}
	In either case, $\tilde\rho(v_{K_1})$ and $\tilde\rho(v_{K_2})$ are adjacent, which establishes the claim. Therefore, we can extend the map $s$ to 1-skeleton of $\cgro$. Since both $\cgro$ and $\widetilde{\cgro}$ are flag simplicial complexes, $s$ extends to a section $\cgro\to\widetilde{\cgro}$; hence, $\cgro$ is simply connected.
\end{proof}

\section{The fundamental groups of the faces of a blowing up}
\label{sec:sal}
Throughout this section, $\ca$ denotes an essential arrangement of linear hyperplanes in a real vector space $A$, $V=A\otimes\mathbb C$, $Z=Z(\ca)$ is the zonotope associated to $\ca$ and $bZ$ denotes the barycentric subdivision of $Z$. We treat $Z$ as an embedded subset of $A$ as indicated in Subsection~\ref{ss:zonotope}. The barycenter of a face $F$ of $Z$ is denoted by $b_F$. Let $\s(\ca)$ be the Salvetti complex associated with $\ca$ as defined in Subsection~\ref{ss:sal}.

\subsection{Salvetti's embedding}\label{ss:embedding}
We recall the definition from \cite{s87} of an embedding  of $\s(\ca)$ into $\cm(\ca\otimes\mathbb C)$. A simplex in $b\s(\ca)$ is represented by a pair $(\Delta,v)$, where $\Delta=[b_{F_0},b_{F_1},\ldots,b_{F_k}]$ is a simplex of $bZ$ where each $F_i$  a face of $Z$, and $v$ is a vertex of $Z$. Write barycentric coordinates of a point $x\in\Delta$ as $x=\sum_{i=0}^k\lambda_{F_i}b_{F_i}$. Then the Salvetti's embedding $\Psi(x,v)$ is defined by 
\begin{equation}
\label{eq:sal embedding}
\Psi:(x,v)\mapsto x+i\left(\sum_{i=0}^k\lambda_{F_i} [p_{F_i}(v) - b_{F_i}]\right),
\end{equation}
where $i=\sqrt{-1}$. 
By Remark~\ref{rmk:intersection}, the definitions of $\Psi$ on each simplex of $b\s(\ca)$ fit together to give a continuous map $\Psi:b\s(\ca)\to V$.

\begin{Theorem}
	\label{thm:salvetti}
	\textup{(cf.~\cite{s87}).} 
	The map $\Psi$ is injective; its  image is contained in $\cm(\ca\otimes \mathbb C)$; and $\Psi:\s(\ca)\to \cm(\ca\otimes\mathbb C)$ is a homotopy equivalence. 
\end{Theorem}

Let $(\Delta,v)$ be as before and let $x\in\Delta$. We also consider a modified version $\Psi'$ of Salvetti's embedding, defined by:
\begin{equation}
\label{eq:modified sal embedding}
\Psi':(x,v)\mapsto x+i\left(\sum_{i=0}^k\lambda_{F_i} p_{F_i}(v)\right).
\end{equation} 

\begin{lemma}
	\label{lem:modified embedding}
	The map $\Psi'$ is an embedding. Let $J$ be the straight line homotopy between $\Psi$ and $\Psi'$. Then the image of $J$ is contained in $\cm(\ca\otimes\mathbb C)$. Hence, $\Psi':b\s(\ca)\to \cm(\ca\otimes\mathbb C)$ is a homotopy equivalence.
\end{lemma}

\begin{proof}
	The first statement follows from the description of intersection of two simplices in $b\s(\ca)$ as in Remark~\ref{rmk:intersection}.
	For the second statement, we follow the argument in \cite{s87}*{pp.\,609-611}. Suppose $x$ is in the relative interior of $\Delta$ and that there is a hyperplane $H\in \ca$ such that $x\in H$. Then $\Delta\subset H$ and $b_{F_i}\in H$ for each $i$. By \cite[Lemma 3]{s87}, the $p_{F_i}(v)$ all lie on the same side of $H$. So,  $\{p_{F_i}(v)-tb_{F_i}\}_{0\le i\le k, 0\le t\le 1}$ lies on the same side of $H$. Hence, the image of $J$ is contained in $\cm(\ca\otimes\mathbb C)$. The final statement of the lemma follows from Theorem~\ref{thm:salvetti}.
\end{proof}

Although both $\Psi$ and $\Psi'$ depend on a choice of points in the fan (i.e., on the choice of the $b_{F_i}$'s), different choices lead to maps $\Psi$ and $\Psi'$ that are homotopic inside $\cm(\ca\otimes\mathbb C)$ via a straight line homotopy. 

Let $\g_x$ and $\Gamma(\ca)$ be as in Subsection~\ref{subsec:deligne groupoid}. By combining Lemma~\ref{lem:modified embedding} with results of \cite{deligne,s87}, one can prove the following theorem.
\begin{Theorem}
	\label{thm:pi1}
	Suppose $\ca$ is a finite, central, real arrangement. Let $x\in\Ga(\ca)$ be a vertex. Consider the embeddings $\Ga(\ca)\to \s(\ca)$ and $\Ga(\ca)\to \cm(\ca\otimes \mathbb C)$ (where the second map is induced by the modified Salvetti's embedding of \eqref{eq:modified sal embedding}). These embeddings induce well-defined homomorphisms $\g_x\to \pi_1(\s(\ca),x)$ and $\g_x\to \pi_1(\cm(\ca\otimes \mathbb C),x)$  both of which  are isomorphisms.
\end{Theorem}

\subsection{Standard subcomplexes and parabolic subgroups}
\label{subsec:standard subcomplexes}

Let $B\in\cq(\ca)$ be a subspace and let $C_B$ be a face of $Z(\ca)$ dual to $B$. The standard complex $K$ of $\s(\ca)$ associated to $C_B$ was defined in Subsection \ref{ss:curve complex groupoid}.  Alternatively, it can be  characterized as the union of all cells of $\s(\ca)$ of form $(C_B,v)$ with $v$ ranging over vertices of $C_B$. Note that $K$ is isomorphic to $\s(\ca_{A/B})$.

Let $\Psi_0:K\to \cm(\ca\otimes\mathbb C)$ be the restriction of the modified Salvetti embedding \eqref{eq:modified sal embedding} to $K$, i.e., if $\Delta=[b_{F_0},b_{F_1},\ldots,b_{F_k}]$ is a simplex of $bC_B$, and if $v$ is a vertex of $C_B$, we have the map,
\begin{equation}
\Psi_0:(x,v)\mapsto x+i\left(\sum_{i=0}^k\lambda_{F_i} p_{F_i}(v)\right).
\end{equation} 
Note that the barycenter $b$ of $C_B$ lies in  $B$. 

We perturb the choice of point in each fan in the definition of $\Psi_0$ to obtain another embedding $\Psi_1:K\to \cm(\ca\otimes\mathbb C)$ such that $b_{F}-b$ is orthogonal to $B$ for any face $F\subset C_B$. Let $J_1$ be the straight line homotopy between $\Psi_0$ and $\Psi_1$. Define $\Psi_2:K\to V=A+iA$ to be the constant map with image $b+ib$. Let $J_2$ be the straight line homotopy between $\Psi_1$ and $\Psi_2$. Let $J$ denote the concatenation of $J_1$ and $J_2$.

The complexifications of the real vector spaces $A$ and $B$ are denoted by $V=A\otimes \cc$ and $E=B\otimes \cc$. The complex vector spaces $V$ and $E$ are identified with $A+iA$ and $B+iB$.  
As before, let $V_\odot$ be the blowup of $V$ along $\ca\otimes \mathbb C$. Let $\db_E V_\odot=E_\odot\times S(V/E)_\odot$ denote the codimension one face of $V_\odot$ associated with $E$.

We claim that $J$ induces a homotopy $\widetilde J:K\times[0,1] \to V_\odot$. Since the interior of $V_\odot$ is identified with $\cm(\ca\otimes \mathbb C)$ and $J(\cdot,t)\in \cm(\ca\otimes\mathbb C)$ for $t<1$, we can define $\widetilde J(\cdot,t)$ to be $J(\cdot,t)$ when $0\le t<1$.
Note that $\Psi_1(x,v)-\Psi_2(x,v)$ is a vector orthogonal to $E$, so this vector defines a point $\mu(x,v)$ in the interior of $S(V/E)_{\odot}$. We can extend $\wt{J}$ to $t=1$ by $\widetilde J((x,v),1):=(b+ib,\mu(x,v))\in \db_E V_\odot=E_\odot\times S(V/E)_\odot$ (note that $b+ib\in \cm(\ca^B\otimes\mathbb C)$, which is the interior $E_\odot$). One checks that $\widetilde J$ is continuous.

Let $B^\perp$ be the orthogonal complement to $B$ in $A$ at $b$. Taking $D$ to be  a small enough open disk in $B^\perp$ centered at $b$,  the arrangement $\ca$ induces an arrangement in $D$ which can be identified with $\ca_{A/B}$. We can assume without loss of generality that the image of $\Psi_1$ is contained in $D\times D$. By Lemma~\ref{lem:modified embedding}, $\Psi_1:K\to (D\times D)- (\bigcup_{H\in \ca}H\times H)$ is a homotopy equivalence. Thus,  the map $K\to S(V/E)_\odot$, defined by $\widetilde J(\cdot, 1)$, is a homotopy equivalence. We summarize the above discussion in the following proposition.

\begin{proposition}
	\label{prop:standard subcomplex homotopy}
	Let $B$ be a subspace of $A$. Let $C_B$ be a face of $Z(\ca)$ dual to $B$ and let $b=bC_B=C_B\cap B$. Suppose $K$ is the standard subcomplex of $\s(\ca)$ associated with $C_B$. Let $\Psi_0:K\to \cm(\ca\otimes\mathbb C)$ be the restriction of the modified Salvetti embedding of \eqref{eq:modified sal embedding}.
	
	Let $D$ be small open disk orthogonal to $B$ at $b$ of complementary dimension. Let $E=B\otimes \mathbb C$ and let $\db_E V_\odot=E_\odot\times S(V/E)_\odot$ be the codimension one face of $V_\odot$ associated with $E$. Then
	\begin{enumerate1}
		\item The map $\Psi_0$ is homotopic in $\cm(\ca\otimes\mathbb C)$ to a map $\Psi_1:K\to (D\times D)- (\bigcup_{H\in \ca}H\times H)$; moreover, $\Psi_1$ is a homotopy equivalence;
		\item The map $\Psi_0$ is homotopic in $V_\odot$ to $\Psi_2:K\to \{x\}\times S(V/E)_\odot$ for some $x\in E_\odot$; moreover, $\Psi_2$ is a homotopy equivalence.
	\end{enumerate1}
\end{proposition}

\subsection{Orthogonal complements of standard subcomplexes}
Let $\ca, A,B,V, E,C_B,$ and $b=b_{C_B}$ be  as in the previous subsection. Suppose $Z=Z(\ca)$ and $Z'=Z(\ca^B)$.
Let $\cp(Z')$ denote the set of faces of $Z'$.  For each $F\in\cp(Z')$ there is a unique face $\hat F$ of $Z$ such that $B\cap \hat F = F$.  If $v$ is a vertex of $Z'$, then $F(v):=\hat v$ is a dual face to $B$.    If $v$, $v'$ are vertices of $Z'$, then $F(v)$ and $F(v')$ are parallel faces of $Z$.  Hence, parallel translation defines a bijection from $\vertex F(v)$ to $\vertex F(v')$.  For each $v\in \vertex Z'$ choose a vertex $u(v)\in F(v)$ so that $u(v)$ corresponds to $u(v')$ under parallel translation. Since we can also treat $b$ as a vertex of $Z'$, the meaning of $u(b)$ is clear.

\begin{lemma}
	\label{lem:proj}
	Let $F$ be a face of $Z'$ and let $v$ be a vertex of $Z'$. Then $u(\prj_F(v))=\prj_{\hat F}(u(v))$. 
\end{lemma}

\begin{proof}
	Let $\ch_F$ (resp. $\ch_{\hat F}$) be the collection of hyperplanes of $\ca$ which intersect an edge of $F$ (resp. $\hat F$) in one point. Then $\ch_F\subset\ch_{\hat F}$. Since each element of $\ch_{\hat F}$ contains $b_{\hat F}=b_F$, each element of $\ch'=\ch_{\hat F}\setminus \ch_F$ contains $B$. By Lemma~\ref{lem:gate} (3), for any $H\in \ch_F$, $v$ and $\prj_F(v)$ are on the same side of $H$. Since no hyperplane of $\ca$ separates $v$ from $u(v)$, $u(v)$ and $u(\prj_F(v))$ are on the same side of $H$ for any $H\in\ch_F$. The same statement holds if $H\in\ca'$, since each element of $\ca'$ contains $B$. The lemma follows after applying Lemma~\ref{lem:gate} (3) to $\hat F$.
\end{proof}

\begin{definition}
	\label{def:phif}
	For each face $F$ of $Z'$,  choose a homeomorphism $\Phi_F: C_B\times F\to \hat F$ such that
	\begin{enumeratei}
		\item For each vertex $v\in Z'$, $\Phi_v$ maps $C_B\times\{v\}$ to the face of $Z'$ which is parallel to $C_B$ and contains $v$ (we regard $Z'$ as a subset of $Z$); moreover, $\Phi_v$ is a simplicial isomorphism induced by parallel translation.
		\item For each edge $e\subset Z'$  endpoints  $v_1$ and $v_2$, $\Phi_e(\{u(b)\}\times e)$ is an edge path in $Z^{(1)}$ of shortest length from $\Phi_{v_1}(\{u(b)\}\times \{v_1\})$ to $\Phi_{v_2}(\{u(b)\}\times \{v_2\})$.
		\item For two faces $F_1\subset F_2$ of $Z'$, the maps $\Phi_{F_1}$ and $\Phi_{F_2}$ agree on $C_B\times F_1$.
		\item We have $\Phi_F(\{b\}\times F)= F$ (where the $F$ on the right-hand side is understood to be a subset of $\hat F$.
	\end{enumeratei} 
	Define $\Ts(\ca^B)$ to be $C_B\times \s(\ca^B)$.
\end{definition}

\begin{lemma}
	\label{lem:Phi}
	The homeomorphisms $\Phi_F$ defined above fit together to induce a map $\Phi:\Ts(\ca^B)\to \s(\ca)$. Moreover, the image of $\Phi$ is $\bigcup_{v\in \vertex Z'}\  (Z,u(v))$.
\end{lemma}

\begin{proof}
	A face of $\Ts(\ca^B)$ is represented by $(F,v)$ where $F$ is a face of $Z'$ and $v$ is a vertex of $F$. Define $\Phi$ on $C_B\times (F,v)$ by taking $C_B\times (F,v)$ to $(\hat F,u(v))$ via $\Phi_F$. Recall that for two cells $(F_1,v_1)$ and $(F_2,v_2)$ in $Z'$, $(F_1,v_1)\subset(F_2,v_2)$ if and only if $F_1\subset F_2$ and $\prj_{F_1}(v_2)=v_1$. So, by Lemma~\ref{lem:proj}, $\hat F_1\subset \hat F_2$ and $\prj_{\hat F_1}(u(v_2))=u(\prj_{F_1}(v_2))=u(v_1)$. Therefore, $(\hat F_1,u(v_1))\subset (\hat F_2,u(v_2))$ as cells of $\s(\ca)$. This shows that $\Phi$ is well-defined. The final statement in the lemma follows from the definition of $\Phi$.
\end{proof}

Let $\Psi_B:\s(\ca^B) \to E$ and $\Psi_A:\s(\ca)\to V$ be two modified Salvetti embeddings (cf. \eqref{eq:modified sal embedding}). Define $f_0$ to be the composition, $$\s(\ca^B)\stackrel{j}{\to} \Ts(\ca^B)\stackrel{k}{\to} \s(\ca)\stackrel{\Psi_A}{\to} V,$$
where $j$ sends $\s(\ca^B)$ to $\{b\}\times \s(\ca^B)$. Perturb the choice of points in each fan in the definition of $\Psi_A$ so that 
\begin{enumerate1}
	\item for each vertex $v\in Z'$, the vector $u(v)-v$ is orthogonal to $B$;
	\item for any two vertices $v_1,v_2\in Z'$, the vectors $u(v_1)-v_1$ and $u(v_2)-v_2$ are parallel. 
\end{enumerate1}
The perturbation of $\Psi_A$ leads to a perturbation of $f_0$. Denote the resulting map by $f_1$. Let $H_1$ be the straight-line homotopy between $f_0$ and $f_1$. 
Define $f_2$ to be the composition, $\s(\ca^B)\stackrel{\Psi_B}{\to} E\hookrightarrow V$, where the second map is the inclusion. Next, we compare $f_1$ and $f_2$.

Recall that any simplex in $b\s(\ca^B)$ has the form $(\Delta,v)$ where $\Delta$ is a simplex in $bZ'$, and $v$ is a vertex of $Z'$. Let $\Delta=[b_{F_0},b_{F_1},\ldots,b_{F_k}]$, where each $F_i$ is a face of $Z'$. Note that $b_{\hat F_i}=b_{F_i}$. Then 
$$
f_2((b_{F_i},v))=b_{F_i} + i(p_{F_i}(v)).
$$
Note that $\Phi$ sends $(\Delta,v)$ to $(\Delta, u(v))$, which is a simplex of $b\s(\ca)$. Thus,
$$
f_1((b_{F_i},v))=b_{\hat F_i} + i(p_{\hat F_i}(u(v)))=b_{F_i} + i(u(p_{F_i}(v))),
$$
where the second inequality follows from Lemma~\ref{lem:proj}.

By our choice of points in each fan, it follows that there is a unit vector $\vec{\mu}\in A$ orthogonal to $B$  such that for each vertex $v\in Z'$,  $u(v)-v$ is the vector $\alpha_v\vec{\mu}$, where $\alpha_v$ is some positive number. So, for each vertex $x$ of $b\s(\ca^B)$,  $H_1(x)-H_0(x)$ is proportional to $i\vec{\mu}$. Since $H_1$ and $H_0$ are linear on each simplex of $b\s(\ca^B)$, the same  holds for any $x\in b\s(\ca^B)$. Let $H_2$ be the straight line homotopy between $f_1$ and $f_2$, and let $H$ be the concatenation of $H_1$ and $H_2$. Note that $H(\cdot,t)\in \cm(\ca\otimes \mathbb C)$ whenever $t<1$.

Then $H(\cdot,t)$ induces a homotopy $\widetilde H:b\s(\ca^B)\times [0,1]\to V_\odot$ as follows. As the interior of $V_\odot$ is identified with $\cm(\ca\otimes \mathbb C)$, put $\widetilde H(x,t)=H(x,t)$ when $t<1$ and define $\widetilde H(x,1)=(\Psi_B(x),\mu)\in \db_E V_\odot=E_\odot\times S(V/E)_\odot$. Again, the interior of $E_\odot$ is equal to $\cm(\ca^B\otimes\mathbb C)$ in which the image of $\Psi_B$ lies, and the point $\mu\in S(V/E)_\odot$ is determined by the vector $i\vec{\mu}$ (note that $i\vec{\mu}\perp E$). One checks that $\widetilde H$ is continuous. Moreover, by Lemma~\ref{lem:modified embedding}, $\widetilde H(\cdot,1):\s(\ca^B)\to E_\odot$ is a homotopy equivalence.

\begin{lemma}
	\label{lem:fit}
	Let $K$ be the standard subcomplex of $\s(\ca)$ associated with $C_B$. Put $K^\perp=\Phi(\{b\}\times \s(\ca^B))$. Then
	\begin{enumerate1}
		\item The restriction of $\Phi$ to $\{b\}\times \s(\ca^B)$ is an embedding.
		\item The intersection $K\cap K^\perp$ is single vertex of $b\s(\ca)$ represented by the point $(b,u(b))$. (Here we treat $b$ as a vertex of $Z'$; hence, $u(b)$ makes sense.)
	\end{enumerate1}
\end{lemma}

\begin{proof}
	Let $\Delta$ be a simplex of $bZ'$. Let $v_1,v_2$ be vertices of $Z'$. It follows from Remark~\ref{rmk:intersection} and Lemma~\ref{lem:proj} that if $(\Delta,v_1)\cap (\Delta,v_2)=(\Delta',v)$, then $(\Delta,u(v_1))\cap (\Delta,u(v_2))=(\Delta',u(v))$. As $\Phi$ maps $(\Delta,v)$ to $(\Delta,u(v))$, statement (1) follows. Note that the unique top-dimensional cell of $K$ that has nonempty intersection with $K^\perp$ is of the form $(C_B,u(b))$. Statement (2) follows.
\end{proof}

\begin{definition}
	The subcomplex $K^\perp$ of $b\s(\ca)$ in the above lemma can be described directly as follows: it is the union of all simplices of form $(\Delta,u(v))$ with $(\Delta,v)$ ranging over simplices of $b\s(\ca^B)$. We call $K^\perp$ the \emph{orthogonal complement} of $K$ at $(b,u(b))$.
\end{definition}

Take a copy of $K$ and a copy of $K^\perp$ and glue them together at $b^*:=(b,u(b))$ to obtain a space $K^*$. So, $K^*$ is the wedge of $b\s(\ca)$ and $b\s(\ca^B)$, and there is a natural embedding $K^*\to \s(\ca)$. Let $g_0:K^*\to V_\odot$ be the map induced by the modified Salvetti embedding $\s(\ca)\hookrightarrow \cm(\ca\otimes\mathbb C)$ (cf.\ \eqref{eq:modified sal embedding}). Let $\widetilde{J},J_1$ and $J_2$ be the homotopies defined in Subsection~\ref{subsec:standard subcomplexes}. Since $J_1$ and $H_1$ arise by perturbing of the choice of points in fans and since this can be done simultaneously and consistently, we can assume that these homotopies fit together to give a homotopy $K^*\times [0,1]\to \cm(\ca\otimes \mathbb C)$. Moreover, the restrictions of $J_2$ and $H_2$ to $(b,u(b))\times [0,1]$ agree (they both give a segment from $(b,u(b))$ to $(b,b)$). By Lemma~\ref{lem:fit} (2), $\widetilde H$ and $\widetilde J$ together induce a homotopy  $K^*\times[0,1]\to V_\odot$ from $g_0$ to $g_1$, where $g_1$ is obtained by gluing together $\widetilde H(\cdot, 1)$ and $\widetilde J(\cdot, 1)$. We summarize the above discussion in the following proposition.

\begin{proposition}
	\label{prop:homotopy}
	Let $B\subset A$ be a subspace in  $\cq(\ca)$. Let $C_B$ be a face of $Z(\ca)$ dual to $B$ and let $K$ be the standard subcomplex of $\s(\ca)$ associated with $C_B$. For a top-dimensional cell $C$ of $K$, let $b^*$ be the barycenter of $C$ and let $K^\perp$ be the orthogonal complement to $K$ at $b^*$. As in the preceding paragraph, glue $K$ to $K^\perp$ at $b^*$ to obtain $K^*$. Then the map $g_0:K^*\to V_\odot$ induced by the modified Salvetti embedding is homotopic (in $V_\odot$) to a continuous map $g_1:K^*\to \db_E V_\odot=E_\odot\times S(V/E)_\odot$ satisfying the following conditions.
	\begin{enumerate1}
		\item There exists a point $y$ in the interior of $E_\odot$ such that $g_1(K^\perp)\subset E_\odot\times\{y\}$ and $(g_1)\vert_{K^{\perp}}:K^\perp\to E_\odot\times\{y\}$ is a homotopy equivalence.
		\item There exists a point $y'$ in the interior of $S(V/E)_\odot$ such that $g_1(K)\subset\{y'\}\times S(V/E)_\odot$ and $(g_1)\vert_{K}:K\to \{y'\}\times S(V/E)_\odot$ is a homotopy equivalence.
	\end{enumerate1}
\end{proposition}

\subsection{The isomorphism between $\calg$ and $\ctop$}
\begin{proposition}
	\label{prop:injective and centralizer}
	Suppose $\ca$ is a finite, central, simplicial real arrangement in a real vector space $A$ and let $V=A\otimes\mathbb C$. Let $B$ be a subspace i $\cq(\ca)$ and let $E=B\otimes \mathbb C$. 
	\begin{enumerate1}
		\item The inclusion map $i:\db_E V_\odot=E_\odot\times S(V/E)_\odot\to V_\odot$ induces an injective map of fundamental groups. Hence, any component of the inverse image of $\db_E V$ in the universal cover $\widetilde{V}_\odot$ is contractible.
		\item Let $U$ be the center of $\pi_1(S(V/E)_\odot)$. Then $i_*(\pi_1(\db_E V_\odot))$ is the centralizer of $i_*(U)$ in $\pi_1(V_\odot)$, and this also is equal to the normalizer of $i_*(U)$ in $\pi_1(V_\odot)$.
	\end{enumerate1}
\end{proposition}

\begin{proof}
	Let $\Phi$ be the map from Lemma~\ref{lem:Phi} and let $\phi:\Ga(\ca^B)\to\Ga(\ca)$ be the map from Definition~\ref{def:Phi}. Let $\Ga(\ca^B)$ and $\Ga(\ca)$ be the graphs defined in Section~\ref{sec:gar}. We identify $\Ga(\ca)$ as the 1-skeleton of $\s(\ca)$ and $\Ga(\ca^B)$ as a subset of $\Ts(\ca^B)=C_B\times \s(\ca^B)$ of the form $\{u(b)\}\times (\s(\ca^B))^{1}$ (where the edges of $\s(\ca)$ and $\s(\ca^B)$ are oriented as in Remark~\ref{rmk:orientation}). It follows from Definition~\ref{def:phif} (2) that we can assume $\phi$ is the restriction of $\Phi$ to $\Ga(\ca^B)$. Theorem~\ref{thm:pi1} implies that the morphism induced by $\Phi\vert_{\{u(b)\}\times \s(\ca^B)}:\{u(b)\}\times \s(\ca^B)\to \s(\ca)$ is described, on the level of fundamental groups,  by a morphism of groupoids $\phi:\g(\ca^B)\to\g(\ca)$. Take a vertex $v$ of $\s(\ca^B)$ and let $x=\Phi((u(b),v))$. Let $K\cong \s(\ca_{A/B})$ be as before. Let $h_1:\pi_1(K,x)\to \pi_1(\s(\ca),x)$ be the map induced by the inclusion $K\to \s(\ca)$. Let $h_2:\pi_1(\s(\ca^B),v)\to \pi_1(\s(\ca),x)$ be the map induced by $\Phi_{|\{u(b)\}\times \s(\ca^B)}$. By Corollary~\ref{cor:injective and centralizer},
	\begin{enumerate1}
		\item $\image h_1$ and $\image h_2$ commute;
		\item $h=h_1\times h_2:\pi_1(K,x)\times \pi_1(\s(\ca^B),v)\to \pi_1(\s(\ca),x)$ is injective;
		\item if $U$ denotes the center of $\pi_1(\s(\ca^B),v)$, then both the centralizer and normalizer of $h_1(U)$ in $\pi_1(\s(\ca),x)$ are generated by $\image h_1$ and $\image h_2$.
	\end{enumerate1}
	The existence of the map $\Phi$ implies that these properties still hold if we choose as base point $\Phi((b,v))$ instead of $x$ in the definition of $h_1$, and let $h_2$ be induced by $\Phi_{|\{b\}\times \s(\ca^B)}$.  Proposition~\ref{prop:homotopy} can be used to complete the proof.
\end{proof}

Let $\calg(\ca)$ be as in Definition~\ref{d:algebraic curve complex}. By Proposition~\ref{prop:standard subcomplex homotopy}~(1),  when $\ca$ is irreducible, $\cgro(\ca)$ is isomorphic $\calg(\ca)$.

\begin{Theorem}
	\label{thm:main body}
	Suppose $\ca$ is the complexification of a finite, central real simplicial arrangement of hyperplanes in an $n$-dimensional vector space and let $l$ be the number of irreducible components of $\ca$. As in Definition~\ref{d:compactcore} let $X$ be the compact core of $V_\odot$ and let $Y$ be the universal cover of $X$.  Then Properties A, C and D in Section~\ref{ss:curvetop} hold and the following statements are true. 
	\begin{enumeratei}
		\item
		The algebraic and topologicial versions of the curve complex $\calg$ and $\ctop$ are identical (and we denote this simplicial complex by $\cac$).
		\item 
		The   faces in  $\db X$ are indexed by the set of simplices in a complex $\ci_0$ which is defined in Definition~\ref{d:I0} in Section~\ref{ss:hyperplanes}.  The faces of $Y$ are indexed by $\cac$ and the group $G=\pi_1(X)$ acts on $\cac$ with quotient space $\ci_0$.
		\item
		The simplicial complex $\cac$ is homotopy equivalent to a wedge of $(n-l-1)$-spheres (where $\dim \cac = n-l-1$).
		\item Each face of $X$ is aspherical and its fundamental group injects into $G$.
		\item
		The fundamental group of each codimensional-one face of $X$ is the normalizer of certain parabolic subgroup.
		\item The stabilizer of each simplex of $\cac$
		is the fundamental group of $\cm(\ca')$ where $\ca'$ is also the complexification of some real simplicial arrangement.
	\end{enumeratei}
\end{Theorem}

\begin{proof}
	Once we verify Properties A, C and D, the theorem follows from Propositions~\ref{prop:iso}, \ref{prop:sc} and \ref{prop:induction}. Property A follows from Proposition~\ref{prop:injective and centralizer}~(1). 
	Property C is follows from Proposition~\ref{prop:injective and centralizer}~(2). Property D follows from Lemma~\ref{lem:conjugate}.
\end{proof}



Michael W. Davis, Department of Mathematics, The Ohio State University, 231 W. 18th Ave., Columbus, Ohio 43210, \url{davis.12@osu.edu}

Jingyin Huang, Department of Mathematics, The Ohio State University, 231 W. 18th Ave., Columbus, Ohio 43210, 
\url{huang.929@osu.edu}

\obeylines
\end{document}